\documentclass[reqno]{amsart}

\input xy
\xyoption{all}
\usepackage{accents}
\usepackage{epsfig}
\usepackage{color}
\usepackage{amsthm}
\usepackage{amssymb}
\usepackage{amsmath}
\usepackage{amscd}
\usepackage{amsopn}
\usepackage{graphicx}

\usepackage{xspace}

\usepackage{url}
\usepackage{enumitem, hyperref}\hypersetup{colorlinks}

\usepackage{color} 

\definecolor{darkred}{rgb}{1,0,0} 
\definecolor{darkgreen}{rgb}{0,0.8,0}
\definecolor{darkblue}{rgb}{0,0,1}

\hypersetup{colorlinks,
linkcolor=darkblue,
filecolor=darkgreen,
urlcolor=darkred,
citecolor=darkgreen}

\makeatletter
\def\reflb#1#2{\begingroup
    #2%
    \def\@currentlabel{#2}%
    \phantomsection\label{#1}\endgroup
}
\makeatother

\numberwithin{equation}{section}
\newtheorem {Theorem}{Theorem}
\numberwithin{Theorem}{section}

\newtheorem {Lemma}[Theorem]    {Lemma}

\newtheorem {Proposition}[Theorem]{Proposition}
\newtheorem {Corollary}[Theorem]{Corollary}
\theoremstyle{definition}

\theoremstyle{remark}
\newtheorem{Remark}[Theorem]{Remark}
\newtheorem{Example}[Theorem]{Example}

\def    \eps    {\epsilon}

\newcommand{\CH}{{\mathcal H}}
\newcommand{\CA}{{\mathcal A}}

\newcommand{\CS}{{\mathcal S}}

\newcommand{\id}{{\mathit id}}
\newcommand{\pt}{{\mathit pt}}
\newcommand{\const}{{\mathit const}}

\newcommand{\tH}{\tilde{H}}

\newcommand{\tA}{\tilde{\mathcal A}}

\newcommand{\PP}{{\mathcal P}}
\newcommand{\bPP}{\bar{\mathcal P}}
\newcommand{\dPP}{\dot{\mathcal P}}

\def    \nat    {{\natural}}
\def    \F      {{\mathbb F}}

\def    \C      {{\mathbb C}}
\def    \R      {{\mathbb R}}

\def    \Z      {{\mathbb Z}}
\def    \N      {{\mathbb N}}
\def    \Q      {{\mathbb Q}}

\def    \CP     {{\mathbb C}{\mathbb P}}

\def    \12    {{\frac{1}{2}}}

\def    \p      {\partial}

\def    \im     {\operatorname{im}}

\def    \HF     {\operatorname{HF}}

\def    \HQ     {\operatorname{HQ}}

\def    \H     {\operatorname{H}}

\def    \CF     {\operatorname{CF}}

\def    \bPP     {\bar{\mathcal{P}}}
\def    \bx     {\bar{x}}
\def    \by     {\bar{y}}
\def    \bz     {\bar{z}}

\def    \MUCZ  {\operatorname{\mu_{\scriptscriptstyle{CZ}}}}
\def    \hMUCZ  {\operatorname{\hat{\mu}_{\scriptscriptstyle{CZ}}}}

\def    \s  {\operatorname{c}}
\def    \hc  {\hat{\operatorname{c}}}

\def    \ssminus        {\smallsetminus}

\begin{document}


\setlength{\smallskipamount}{6pt}
\setlength{\medskipamount}{10pt}
\setlength{\bigskipamount}{16pt}





\title[Conley Conjecture Revisited]{Conley Conjecture Revisited}

\author[Viktor Ginzburg]{Viktor L. Ginzburg}
\author[Ba\c sak G\"urel]{Ba\c sak Z. G\"urel}

\dedicatory{\normalsize{Dedicated to Dusa McDuff on the occasion of 
her 70th birthday}}

\address{BG: Department of Mathematics, University of Central Florida,
  Orlando, FL 32816, USA} \email{basak.gurel@ucf.edu}

\address{VG: Department of Mathematics, UC Santa Cruz, Santa Cruz, CA
  95064, USA} \email{ginzburg@ucsc.edu}



\date{\today} 

\thanks{The work is partially supported by NSF CAREER award
  DMS-1454342 (BG) and NSF grants DMS-1414685 (BG) and DMS-1308501
  (VG)}


\begin{abstract} We show that whenever a closed symplectic manifold
  admits a Hamiltonian diffeomorphism with finitely many simple
  periodic orbits, the manifold has a spherical homology class of
  degree two with positive symplectic area and positive integral of
  the first Chern class. This theorem encompasses all known cases of
  the Conley conjecture (symplectic CY and negative monotone
  manifolds) and also some new ones (e.g., weakly exact symplectic
  manifolds with non-vanishing first Chern class).

  The proof hinges on a general Lusternik--Schnirelmann type result
  that, under some natural additional conditions, the sequence of mean
  spectral invariants for the iterations of a Hamiltonian
  diffeomorphism never stabilizes. We also show that for the
  iterations of a Hamiltonian diffeomorphism with finitely many
  periodic orbits the sequence of action gaps between the ``largest''
  and the ``smallest'' spectral invariants remains bounded and, as a
  consequence, establish some new cases of the $C^\infty$-generic
  existence of infinitely many simple periodic orbits.
  \end{abstract}

\maketitle


\tableofcontents


\section{Introduction and main results}
\label{sec:intro+results}

\subsection{Introduction}
\label{sec:intro}

We prove a Conley conjecture type theorem encompassing all known cases
of the conjecture and also covering some new ones. An essential new
feature of the theorem and its proof is a direct connection between
the Conley conjecture and properties of the symplectic form, whereas
all previous results utilized either only the first Chern class (the
CY case) or the interplay between the two cohomology classes (the
negative monotone case). The key to the proof is a general result
concerning the behavior of the sequence of the mean spectral
invariants for iterations of a Hamiltonian diffeomorphism. Namely, we
show that the sequence never stabilizes when the periodic orbits are
isolated and none of the orbits is a symplectically degenerate maximum
(SDM).  We also further investigate Hamiltonian diffeomorphisms with
finitely many simple periodic orbits. In particular, we relate mean
spectral invariants of such maps to resonance relations for augmented
actions and show that the sequence of certain action gaps remains
bounded. As a consequence, we prove the generic existence of
infinitely many simple periodic orbits in some new cases.

Let us now review the context and background for these results in more
detail.

The Conley conjecture asserts that for a broad class of closed
symplectic manifolds every Hamiltonian diffeomorphism has infinitely
many simple periodic orbits. The conjecture has been established for
all symplectic CY manifolds, \cite{GG:gaps,He:irr}, and all negative
monotone symplectic manifolds, \cite{CGG,GG:nm}; see also
\cite{FH,Gi:CC,Hi,SZ} for some relevant results and \cite{GG:survey}
for a general survey and further references. In this paper, we prove
that when a closed symplectic manifold $(M,\omega)$ admits a
Hamiltonian diffeomorphism with finitely many periodic orbits, there
is a class $A\in\pi_2(M)$ with $\omega(A)>0$ and
$\left<c_1(TM),A\right>>0$. This result implies the Conley conjecture
for CY and negative monotone manifolds. Furthermore, it also shows
that the Conley conjecture holds, for instance, for weakly exact
symplectic manifolds $(M,\omega)$ (i.e., such that
$\omega\!\mid_{\pi_2(M)}=0$) with $c_1(TM)\!\mid_{\pi_2(M)}\neq 0$.
(We refer the reader to \cite{Go} for a construction of such
manifolds.)

The Conley conjecture fails for $S^2$, $\CP^n$, complex Grassmannians
and, in fact, for all closed symplectic manifolds admitting
Hamiltonian circle or torus actions with isolated fixed
points. Indeed, then a generic element of the torus gives rise to a
Hamiltonian diffeomorphism with finitely many periodic orbits.  In
fact, all known manifolds for which the Conley conjecture fails admit
Hamiltonian $S^1$-actions with isolated fixed points. However, these
manifolds may have other types of Hamiltonian diffeomorphisms with
finitely many periodic orbits. This is the case when $M=S^2$ and
$M=(S^2)^n$ and hypothetically for all symplectic manifolds with such
$S^1$-actions. For $S^2$ these diffeomorphisms are the so-called
pseudo-rotations, which play a prominent role in low-dimensional
dynamics (see, e.g., \cite{AK,FK}), and in general one can view
Hamiltonian diffeomorphisms with finitely many periodic orbits as
higher-dimensional analogs of pseudo-rotations.

Perhaps the most comprehensive conjecture identifying the manifolds
for which the Conley conjecture fails is due to Chance and McDuff. It
asserts that such a manifold has a non-vanishing GW invariant or even
a non-trivially deformed quantum product.  This conjecture fits well
with the examples discussed above; for every symplectic manifold with
Hamiltonian $S^1$-action is in a certain sense uniruled and thus has a
non-vanishing GW invariant; see \cite{McD:unir}. Our main result
provides further evidence supporting the Chance--McDuff
conjecture. Indeed, it implies, in particular, that
$\omega\!\mid_{\pi_2(M)}\neq 0$ whenever the Conley conjecture fails,
and hence the manifold can at least have non-zero GW
invariants. Additional evidence comes from the results on the bounded
action gap discussed below.

Yet not a single particular case of the Chance--McDuff conjecture has
been proved. The difficulty lies in identifying a source of
holomorphic curves. Indeed, while the effect of holomorphic curves on
Hamiltonian dynamics is well understood, it is completely unknown how
to detect the existence of holomorphic curves from the dynamical
behavior (e.g., periodic orbits) of Hamiltonians.

The method used in the proof of the main theorem is quite different
from the proofs of other Conley conjecture type results. Namely, the
main new ingredient is a Lusternik--Schnirelmann type result in the
spirit of \cite[Prop.\ 6.2]{GG:gaps} and \cite[Thm.\ 1.1]{GG:convex},
which might be of independent interest. To state this result, assume
that all periodic orbits of a Hamiltonian diffeomorphism $\varphi_H$
of a rational symplectic manifold $M$ are isolated and none of these
orbits is an SDM. Then we show that the sequence of the mean spectral
invariants $\hc_k:=\s_{[M]}(H^{\nat ^k})/k$ associated with the
fundamental class of $M$ for the iterations $\varphi_H^k$ never
stabilizes and, in fact, $\hc_k>\hc_\infty:=\lim \hc_k$. The key point
here is that the inequality is strict; the non-strict inequality is an
easy consequence of the standard properties of spectral invariants and
holds under no additional assumptions on $H$. (Furthermore, one can
also use this result to prove the negative monotone case of the Conley
conjecture.)

We also investigate Hamiltonian diffeomorphisms with finitely many
periodic orbits of monotone symplectic manifolds. We show that for
such a Hamiltonian diffeomorphism the sequence of action gaps
$\s_{[M]}\big(H^{\nat ^k}\big)-\s_{[\pt]}\big(H^{\nat ^k}\big)$
remains bounded as $k\to\infty$. Usually such upper bounds result from
non-trivial relations in the quantum homology of $M$. Hence this
theorem can also be viewed as further evidence supporting the
Chance--McDuff conjecture. We also relate the limit $\hc_\infty$ to
the augmented action of periodic orbits, refine the action--index
resonance relations from \cite{CGG,GG:gaps} and, as a consequence,
obtain a new $C^\infty$-generic existence result for infinitely many
simple periodic orbits.

\subsection{Main results}
\label{sec:results}
Let us now state the main theorems and outline the organization of the
paper.

The conventions and notation used in the paper and the necessary
preliminary material on filtered and local Floer homology, the
pair-of-pants product, spectral invariants and action carriers are
reviewed in Section \ref{sec:prelim}. Most of the definitions and
facts stated there are quite standard, although the
Lusternik--Schnirelmann inequality for the pair-or-pants product
(Proposition \ref{prop:LS}) might be new. Here we only note that we
focus exclusively on contractible periodic orbits, i.e., a ``periodic
orbit'' means a ``contractible periodic orbit.'' Our key Conley
conjecture type result is the following.

\begin{Theorem}
\label{thm:main0}
Assume that a closed symplectic manifold $M$ admits a Hamiltonian
diffeomorphism $\varphi_H$ with finitely many periodic orbits. Then
there exists $A\in\pi_2(M)$ such that $\omega(A)>0$ and
$\left<c_1(TM),A\right>>0$.
\end{Theorem}

The main new point of this theorem is the existence of a spherical
class $A$ satisfying the first of these two conditions. Theorem
\ref{thm:main0} is proved in Section \ref{sec:pf}. The proof hinges on
the following general fact established in Section \ref{sec:mean}.

\begin{Theorem}[Lusternik--Schnirelmann inequality for mean spectral
  invariants]
\label{thm:c-infty0}
Assume that $M$ is rational, all periodic orbits of $H$ are isolated
and none of the orbits is an SDM. Then
$$
\hc_k>\hc_\infty
$$
for all $k$, where $\hc_k=\s_{[M]}\big(H^{\nat k}\big)/k$ is the mean
spectral invariant associated with the fundamental class and applied
to the ``$k$-iterated'' Hamiltonian $H^{\nat k}$, and
$\hc_\infty=\lim \hc_k$.
\end{Theorem}

Finally, in Section \ref{sec:further} we turn to Hamiltonians $H$ with
finitely many simple periodic orbits. By passing to an iteration, we
can always assume that every simple periodic orbit of $H$ is
one-periodic, i.e., $H$ is perfect.

\begin{Theorem}[Action Gap]
\label{thm:gap-bound0}
Assume that $H$ is perfect and $M$ is monotone with monotonicity
constant $\lambda$. Then
\begin{equation}
\label{eq:gap-bound}
\s_{[M]}\big(H^{\nat k}\big)-\s_{[\pt]}\big(H^{\nat k}\big)\leq 2\lambda n,
\end{equation}
for all but possibly a finite number of iterations $k\in\N$.
\end{Theorem}

We also show in Theorem \ref{thm:aa-c_infty} that for a perfect
Hamiltonian $H$, the asymptotic mean spectral invariant
$\hc_\infty(H)$ is equal to the augmented action of its action
carrier. As a consequence, we refine the action--index resonance
relations from \cite{CGG,GG:gaps} by proving the existence of several
geometrically distinct simple orbits with augmented actions equal
$\hc_\infty(H)$; see Corollaries \ref{cor:res1}, \ref{cor:res2} and
\ref{cor:CPn}. As another application, in Corollary \ref{cor:gen} we
establish $C^\infty$-generic existence of infinitely many simple
periodic orbits for essentially all closed monotone symplectic
manifolds with a minor hypothetical exception; cf.\ \cite{GG:generic}.

\medskip
\noindent{\bf Acknowledgements.} The authors are grateful to Roman
Golovko, Kaoru Ono, Sobhan Seyfaddini and Michael Usher for useful
discussions. Parts of this work were carried out while both of the
authors were visiting RIMS (Kyoto, Japan), TFC (Sendai, Japan), IMBM
(Istanbul, Turkey) and during the first author's visit to NCTS
(Taipei, Taiwan). The authors would like to thank these institutes for
their warm hospitality and support.

\section{Preliminaries}
\label{sec:prelim}
The goal of this section is to set notation and conventions and to
give a brief review of Floer homology and several other notions used
in the paper.

\subsection{Conventions and notation}
\label{sec:conv}
Let $(M^{2n},\omega)$ be a closed symplectic manifold. Throughout the
paper we will usually assume that $M$ is \emph{rational}, i.e., the
group $\left<[\omega], {\pi_2(M)}\right>\subset\R$ formed by the
integrals of $\omega$ over the spheres in $M$ is discrete. This
condition is obviously satisfied when $M$ is \emph{negative monotone}
or \emph{monotone}, i.e.,
$[\omega]\!\mid_{\pi_2(M)}=\lambda c_1(TM)\!\mid_{\pi_2(M)}$ for some
$\lambda<0$ in the former case or $\lambda\geq 0$ in the latter.  When
$\lambda>0$, we will sometimes say that $M$ is \emph{positive
  monotone}. The positive generator $N$ of the group
$\left<c_1(TM), {\pi_2(M)}\right>\subset\Z$ is called the
\emph{minimal Chern number} of $M$. (When this group is zero, we set
$N=\infty$.)  Recall also that $M$ is \emph{symplectic Calabi--Yau
  (CY)} if $c_1(TM)\!\mid_{\pi_2(M)}=0$ and $M$ is called
\emph{symplectically aspherical} when, in addition,
$[\omega]\mid_{\pi_2(M)}=0$.

All Hamiltonians $H$ considered in this paper are assumed to be
$k$-periodic in time, i.e., $H\colon S^1_k\times M\to\R$, where
$S^1_k=\R/k\Z$ and $k\in\N$.  When the period $k$ is not specified, it
is equal to one and $S^1=S^1_1=\R/\Z$. We set $H_t = H(t,\cdot)$ for
$t\in S^1_k$. The Hamiltonian vector field $X_H$ of $H$ is defined by
$i_{X_H}\omega=-dH$. The (time-dependent) flow of $X_H$ is denoted by
$\varphi_H^t$ and its time-one map by $\varphi_H$. Such time-one maps
are referred to as \emph{Hamiltonian diffeomorphisms}.  A one-periodic
Hamiltonian $H$ can always be treated as $k$-periodic, which we will
then denote by $H^{\nat k}$ and, abusing terminology, call
$H^{\nat k}$ the $k$th iteration of $H$.

Let $H$ and $K$ be one-periodic Hamiltonians such that $H_1=K_0$
together with $t$-derivatives of all orders. We denote by $H \nat K$
the two-periodic Hamiltonian equal to $H_t$ for $t\in [0,\,1]$ and
$K_{t-1}$ for $t\in [1,\,2]$. Thus $H^{\nat k}=H \nat \ldots \nat H$
($k$ times). More generally, when $H$ is $l$-periodic and $K$ is
$k$-periodic, $H\nat K$ is $(l+k)$-periodic. (Strictly speaking, here
we need to assume that $H_l=K_0$ again together with all
$t$-derivatives.)

Let $x\colon S^1_k\to M$ be a contractible loop. A \emph{capping} of
$x$ is an equivalence class of maps $A\colon D^2\to M$ such that
$A\mid_{S^1_k}=x$. Two cappings $A$ and $A'$ of $x$ are equivalent if
the integrals of $\omega$ and $c_1(TM)$ over the sphere obtained by
attaching $A$ to $A'$ are equal to zero. A capped closed curve
$\bar{x}$ is, by definition, a closed curve $x$ equipped with an
equivalence class of cappings. In what follows, the presence of
capping is always indicated by a bar.

The action of a Hamiltonian $H$ on a capped closed curve
$\bar{x}=(x,A)$ is
$$
\CA_H(\bar{x})=-\int_A\omega+\int_{S^1} H_t(x(t))\,dt.
$$
The space of capped closed curves is a covering space of the space of
contractible loops, and the critical points of $\CA_H$ on this space
are exactly the capped one-periodic orbits of $X_H$. The \emph{action
  spectrum} $\CS(H)$ of $H$ is the set of critical values of
$\CA_H$. This is a zero measure set; see, e.g., \cite{HZ}. When $M$ is
rational, $\CS(H)$ is a closed, and hence nowhere dense,
set. Otherwise, $\CS(H)$ is dense in $\R$.

These definitions extend to $k$-periodic orbits and Hamiltonians in an
obvious way. Clearly, the action functional is homogeneous with
respect to iteration:
$$
\CA_{H^{\nat k}}\big(\bx^k\big)=k\CA_H(\bx),
$$
where $\bx^k$ is the $k$th iteration of the capped orbit $\bx$.

As mentioned in the introduction, all of our results concern only
contractible periodic orbits and throughout the paper \emph{a periodic
  orbit is always assumed to be contractible, even if this is not
  explicitly stated}.  We denote the set of $k$-periodic orbits of $H$
by $\PP_k(H)$. The set of all periodic orbits will be denoted by
$\PP(H)$. Finally, we will write $\dot{\PP}_k(H)$ and $\dot{\PP}(H)$
for the collections of \emph{simple} (i.e., not iterated) periodic
orbits of $H$.

A periodic orbit $x$ of $H$ is said to be \emph{non-degenerate} if the
linearized return map $d\varphi_H \colon T_{x(0)}M\to T_{x(0)}M$ has
no eigenvalues equal to one. Following \cite{SZ}, we call $x$
\emph{weakly non-degenerate} if at least one of the eigenvalues is
different from one and \emph{totally degenerate} if all eigenvalues
are equal to one. A Hamiltonian $H$ is (weakly) non-degenerate if all
its one-periodic orbits are (weakly) non-degenerate and $H$ is
\emph{strongly non-degenerate} if all periodic orbits of $H$ (of all
periods) are non-degenerate.

Let $\bar{x}=(x,A)$ be a non-degenerate capped periodic orbit.  The
\emph{Conley--Zehnder index} $\MUCZ(\bar{x})\in\Z$ is defined, up to a
sign, as in \cite{Sa,SZ}. (Sometimes, we will also use the notation
$\MUCZ(x,A)$.)  In this paper, we normalize $\MUCZ$ so that
$\MUCZ(\bar{x})=n$ when $x$ is a non-degenerate maximum (with trivial
capping) of an autonomous Hamiltonian with small Hessian. With this
normalization, the Conley--Zehnder index is the negative of that in
\cite{Sa}. The \emph{mean index} $\hMUCZ(\bx)\in\R$ measures, roughly
speaking, the total angle swept by certain unit eigenvalues of the
linearized flow $d\varphi^t_H\mid_x$ with respect to the
trivialization associated with the capping; see \cite{Lo,SZ}. The mean
index is defined even when $x$ is degenerate and depends continuously
on $H$ and $\bx$ in the obvious sense.  Furthermore,
\begin{equation}
\label{eq:mean-cz}
\big|\hMUCZ(\bx)-\MUCZ(\bx)\big|\leq n
\end{equation}
and the inequality is strict when $x$ is weakly non-degenerate.  The
mean index is homogeneous with respect to iteration:
\begin{equation}
\label{eq:mean-hom}
\hMUCZ\big(\bx^k\big)=k\hMUCZ(\bx).
\end{equation}

\subsection{Floer homology}
In this subsection, we very briefly discuss, mainly to set notation,
the construction of filtered Floer homology. We refer the reader to,
e.g., \cite{FO, GG:gaps, HS, MS, Sa, SZ} for detailed accounts and
additional references. We also recall the definition of the local
Floer homology.

\subsubsection{Filtered Floer homology}
Fix a ground field $\F$. Let $H$ be a non-degenerate Hamiltonian on
$M$. Denote by $\CF^{(-\infty,\, b)}_m(H)$, with
$b\in (-\infty,\,\infty]\setminus\CS(H)$, the vector space of formal
linear combinations
$$ 
\sigma=\sum_{\bar{x}\in \bPP(H)} \alpha_{\bar{x}}\bar{x}, 
$$
where $\alpha_{\bar{x}}\in\F$ and $\MUCZ(\bar{x})=m$ and
$\CA_H(\bar{x})<b$. Furthermore, we require, for every $a\in \R$, the
number of terms in this sum with $\alpha_{\bar{x}}\neq 0$ and
$\CA_H(\bar{x})>a$ to be finite. The graded $\F$-vector space
$\CF^{(-\infty,\, b)}_*(H)$ is endowed with the Floer differential
counting the $L^2$-anti-gradient trajectories of the action
functional. Thus we obtain a filtration of the total Floer complex
$\CF_*(H):=\CF^{(-\infty,\, \infty)}_*(H)$. Furthermore, set
$$
\CF^{(a,\, b)}_*(H):=\CF^{(-\infty,\,
  b)}_*(H)/\CF^{(-\infty,\,a)}_*(H),
$$
where $-\infty\leq a<b\leq\infty$ are not in $\CS(H)$. The resulting
homology, the \emph{filtered Floer homology} of $H$, is denoted by
$\HF^{(a,\, b)}_*(H)$ and by $\HF_*(H)$ when
$(a,\,b)=(-\infty,\,\infty)$.  Working with filtered Floer homology,
\emph{we will always assume that the end points of the action interval
  are not in the action spectrum.} The degree of a class
$w\in \HF^{(a,\, b)}_*(H)$ is denoted by~$|w|$.

Over $\F=\Z_2$, we can view a chain
$\sigma=\sum \alpha_{\bx}\bx\in \CF_*(H)$ as simply a collection of
capped one-periodic orbits $\bx$ for which $\alpha_{\bx}\neq 0$. In
general, we will say that $\bx$ \emph{enters} the chain $\sigma$ when
$\alpha_{\bx}\neq 0$. Note also that it is often convenient to view
$\CF^{(a,\, b)}_*(H)$ as a subspace, but not in general a subcomplex,
of $\CF_*(H)$.

The total Floer complex and homology are modules over the
\emph{Novikov ring} $\Lambda$. In this paper, the latter is defined as
follows. Set
$$
I_\omega(A)=-\omega(A)\text{ and } I_{c_1}(A)=-2\left<c_1(TM),
  A\right>,
$$
where $A\in\pi_2(M)$.  Thus
\begin{equation}
\label{eq:omega-c1}
I_\omega=\frac{\lambda}{2}I_{c_1}
\end{equation}
when $M$ is monotone or negative monotone. 

Let $\Gamma$ be the quotient of $\pi_2(M)$ by the equivalence relation
where the two spheres $A$ and $A'$ are considered to be equivalent if
$I_\omega(A)=I_\omega(A')$ and $I_{c_1}(A)=I_{c_1}(A')$. In other
words,
$$
\Gamma=\frac{\pi_2(M)}{\ker I_\omega\cap \ker I_{c_1}}.
$$
For instance, $\Gamma\simeq \Z$ when $M$ is negative monotone or
monotone with $\lambda\neq 0$. The homomorphisms $I_\omega$ and
$I_{c_1}$ descend to $\Gamma$.

The group $\Gamma$ acts on $\CF_*(H)$ and on $\HF_*(H)$ by recapping:
an element $A\in \Gamma$ acts on a capped one-periodic orbit $\bar{x}$
of $H$ by attaching the sphere $A$ to the original capping. Denoting
the resulting capped orbit by $\bx\# A$, we have
$$
\MUCZ(\bx\# A)=\MUCZ(\bx)+ I_{c_1}(A)
$$
when $x$ is non-degenerate. In a similar vein,
\begin{equation}
\label{eq:delta}
\CA_H(\bx\# A)=\CA_H(\bx)+I_\omega(A)
\text{ and } \hMUCZ(\bx\# A)=\hMUCZ(\bx)+ I_{c_1}(A)
\end{equation}
regardless of whether $x$ is non-degenerate or not.

The Novikov ring $\Lambda$ is a certain completion of the group ring
of $\Gamma$ over $\F$. Namely, $\Lambda$ comprises formal linear
combinations $\sum \alpha_A e^A$, where $\alpha_A\in\F$ and
$A\in \Gamma$, such that for every $a\in \R$ the sum contains only
finitely many terms with $I_\omega(A) > a$ and $\alpha_A\neq 0$. The
Novikov ring $\Lambda$ is graded by setting $|e^A|=I_{c_1}(A)$.  The
action of $\Gamma$ turns $\CF_*(H)$ and $\HF_*(H)$ into
$\Lambda$-modules.

When $M$ is rational, the definition of Floer homology extends to all,
not necessarily non-degenerate, Hamiltonians by continuity.  Namely,
let $H$ be an arbitrary (one-periodic in time) Hamiltonian on $M$ and
let the end-points $a$ and $b$ of the action interval be outside
$\CS(H)$. By definition, we set
\begin{equation}
\label{eq:deg-filt}
\HF^{(a,\, b)}_*(H)=\HF^{(a,\, b)}_*(\tH),
\end{equation}
where $\tH$ is a non-degenerate, small perturbation of $H$. It is easy
to see that the right hand side of \eqref{eq:deg-filt} is independent
of $\tH$ once $\tH$ is sufficiently close to $H$. (The assumption that
$M$ is rational is essential at this point; for, otherwise, the right
hand side of \eqref{eq:deg-filt} depends on the perturbation $\tH$
even when $\tH$ is arbitrarily close to $H$. We refer the reader to
\cite{He:irr} and also to \cite[Remark 2.3]{GG:gaps} for the
definition in the irrational case.)

The total Floer homology is independent of the Hamiltonian and, up to
a shift of the grading and the effect of recapping, is isomorphic to
the homology of $M$. More precisely, we have
$$
\HF_*(H)\cong \H_ {*+n}(M;\F)\otimes \Lambda
$$
as graded $\Lambda$-modules.

\begin{Remark}
\label{rmk:Floer}
In general, in order for the Floer homology to be defined, certain
regularity conditions must be satisfied generically. To ensure this,
one has to either require $M$ to be weakly monotone (see
\cite{HS,MS,Ono:AC,Sa}) or utilize the machinery of virtual cycles
(see \cite{FO,FOOO,LT} or, for the polyfold approach,
\cite{HWZ:SC,HWZ:poly} and references therein). In the latter case,
the ground field $\F$ is required to have zero characteristic. Most of
the proofs in this paper do not rely on the machinery of virtual
cycles. To be more specific, the proof of Theorem \ref{thm:main0}
comprises known and new cases of the Conley conjecture. In the proof
of the new cases (Proposition \ref{prop:omega}), arguing by
contradiction, one can assume that $I_\omega=0$ and hence the Floer
homology can be defined without virtual cycles.
\end{Remark}

\subsubsection{Local Floer homology}
\label{sec:LFH}
The notion of local Floer homology goes back to the original work of
Floer and it has been revisited a number of times since then. Here we
only briefly recall the definition following mainly
\cite{Gi:CC,GG:gaps,GG:gap} where the reader can find a much more
thorough discussion and further references.

Let $\bx=(x,A)$ be a capped isolated one-periodic orbit of a
Hamiltonian $H\colon S^1\times M\to \R$. Pick a sufficiently small
tubular neighborhood $U$ of $x$ and consider a non-degenerate
$C^2$-small perturbation $\tH$ of $H$.  (Strictly speaking $U$ should
be a neighborhood of the graph of the orbit in the extended phase
space $S^1\times M$.)  The orbit $x$ splits into non-degenerate
one-periodic orbits $x_j$ of $\tH$, which are $C^1$-close to $x$. The
capping of $\bx$ gives rise to a capping of $x_j$ and
$\CA_{\tH}(\bx_j)$ is close to $\CA_H(\bx)$.

Every Floer trajectory between the orbits $\bx_j$ is contained in $U$,
provided that $\|\tH-H\|_{C^2}$ is small enough.  Thus, by the
compactness and gluing theorems, every broken anti-gradient trajectory
connecting two such orbits also lies entirely in $U$. Similarly to the
definition of the ordinary Floer homology, consider the complex
$\CF_*(\tH,\bx)$ over $\F$ generated by the capped orbits $\bx_j$,
graded by the Conley--Zehnder index and equipped with the Floer
differential defined in the standard way.  The continuation argument
shows that the homology of this complex is independent of the choice
of $\tH$ and of other auxiliary data (e.g., an almost complex
structure). We refer to the resulting homology group, denoted by
$\HF_*(\bx)$ or $\HF_*(x,A)$, as the \emph{local Floer homology} of
$\bx$. For instance, if $x$ is non-degenerate and $\MUCZ(\bx)=m$, we
have $\HF_l(\bx)=\F$ when $l=m$ and $\HF_l(\bx)=0$ otherwise.

The above construction is local: it requires $H$ to be defined only on
a neighborhood of $x$ and the capping of $x$ is used only to fix a
trivialization of $TM|_x$ and hence an absolute $\Z$-grading of
$\HF_*(\bx)$.

By definition, the \emph{support} of $\HF_*(\bx)$ is the collection of
integers $m$ such that $\HF_m(\bx)\neq 0$. By \eqref{eq:mean-cz} and
continuity of the mean index, $\HF_*(\bx)$ is supported in the
interval $[\hMUCZ(\bx)-n,\, \hMUCZ(\bx)+n]$. Moreover, when $x$ is
weakly non-degenerate, the closed interval can be replaced by an open
interval.

Recall that a capped isolated periodic orbit $\bx$ is called a
\emph{symplectically degenerate maximum (SDM)} if
$\HF_{\hMUCZ(\bx)+n}(\bx)\neq 0$, where we set $\HF_*(\bx)=0$ when
$*\not\in\Z$.  This property is independent of the capping. An SDM
orbit is necessarily totally degenerate and an iteration of an SDM is
again an SDM; see, e.g., \cite{Gi:CC,GG:gap}. By \cite[Thm.\
1.18]{GG:gaps}, a Hamiltonian diffeomorphism of a rational symplectic
manifold with an SDM orbit has infinitely many periodic orbits.

\subsection{The pair-of-pants product}
\label{sec:product}
In this section, we briefly recall several properties of the
pair-of-pants product in Floer homology, referring the reader to,
e.g., \cite{AS,MS,PSS} for more detailed accounts.

The filtered Floer homology carries the so-called \emph{pair-of-pants}
product. On the level of complexes, this product, which we denote by
$*$, is a map
\begin{equation}
\label{eq:prod-c}
\CF^{(-\infty,\, a)}_m(H)\otimes \CF^{(-\infty,\, b)}_l(K)
\to \CF^{(-\infty,\, a+b)}_{m+l-n}(H\nat K)
\end{equation}
giving rise on the level of homology to an associative,
graded-commutative product
\begin{equation}
\label{eq:prod-h}
\HF^{(-\infty,\, a)}_m(H)\otimes \HF^{(-\infty,\, b)}_l(K)
\to \HF^{(-\infty,\, a+b)}_{m+l-n}(H\nat K).
\end{equation}
The product turns the direct sum of the total Floer homology
$$
\bigoplus_{k\geq 0} \HF_*\big(H^{\nat k}\big)
$$
into an associative and graded-commutative unital algebra, where
$\HF_*\big(H^{\nat 0}\big)$ is by definition the quantum homology
$\HQ_*(M)$ of $M$. This direct sum is isomorphic, as an algebra, to
the algebra of polynomials with coefficients in $\HQ_*(M)$. The unit
in this algebra is the fundamental class of $M$ in $\HQ_*(M)$.

On the level of complexes the product \eqref{eq:prod-c} is not
associative in any sense. Furthermore, in order for the product to be
defined the Hamiltonians $H$ and $K$ must meet certain generic
regularity conditions. With this in mind, the product on the level of
homology is defined ``by continuity'' for all Hamiltonians, at least
when $M$ is rational.

Note that the multiplication by the fundamental class
$[M]\in \HF_n(H)$ is a grading-preserving (but not action-preserving)
isomorphism
\begin{equation}
\label{eq:*M}
* [M]\colon \HF_*\big(H^{\nat k}\big)\stackrel{\cong}{\longrightarrow} 
\HF_*\big(H^{\nat (k+1)}\big).
\end{equation}

A word of warning is due regarding the behavior of the action
filtration with respect to the pair-of-pants product. The original
definition of the product as in, e.g., \cite{MS,PSS} relies on cutting
off the Hamiltonians. With this definition the action filtration is
not preserved in the sense of \eqref{eq:prod-c} and
\eqref{eq:prod-h}. Cutting off can be avoided when $H_t$ and $K_t$
vanish for $t$ close to $0\in S^1$, which can always be achieved for
closed manifolds by simply reparametrizing the Hamiltonians. Under
this extra assumption, \eqref{eq:prod-c} and \eqref{eq:prod-h} hold as
stated. However, a more elegant definition of the pair-of-pants
product is given in \cite{AS}, where the domain $\Sigma$ of a
pair-of-pants curve $u$ is treated as a double cover of the cylinder,
branching at one point. The domain $\Sigma$ naturally carries the
``coordinates'' $(s,t)$ lifted from the cylinder, which are true
coordinates on the three open half-cylindrical parts of the domain,
and on each of these parts $u$ satisfies the Floer equation for the
corresponding Hamiltonian $H$ or $K$ or $H\nat K$. For a pair-of-pants
curve $u$ connecting $\bx$ and $\by$ to $\bz$, we have
\begin{equation}
\label{eq:E}
\CA_{H}(\bx)+\CA_{K}(\by)-\CA_{H\nat K}(\bz)=E(u),\textrm{ where } E(u):=\int_\Sigma \big\|\p_s u\big\|^2
\,ds\,dt\geq 0,
\end{equation}
which, in particular, implies \eqref{eq:prod-c} and \eqref{eq:prod-h};
see \cite[Eq.\ (3-18)]{AS}. Note that here the capping of $\bz$ is
obtained by attaching $u$ to the cappings of $\bx$ and $\by$.

Clearly, $\CA_{H}(\bx)+\CA_{K}(\by)=\CA_{H\nat K}(\bz)$ if and only if
$E(u)=0$, and if and only if $u$ is a ``trivial'' pair-of-pants
curve. More specifically, then $\p_s u\equiv 0$ and $x(0)=y(0)$, and
thus $u$ is the projection from $\Sigma$ to the figure-8 curve formed
by $x$ and $y$.

As a consequence of \eqref{eq:E}, there is a variant of the
Lusternik-Schnirelmann inequality in the spirit of \cite[Prop.\
6.2]{GG:gaps} and \cite[Thm.\ 1.1]{GG:convex} for the pair-of-pants
product, which plays an important role in the proof of Theorem
\ref{thm:main0}. For the sake of brevity, we state here only a simple
version of this inequality sufficient for our purposes and concerning
iterated Hamiltonians.

Let $H$ be a one-periodic Hamiltonian such that all its one-periodic
orbits $x_i$ are isolated. Let us fix small disjoint neighborhoods
$U_i$ of these orbits. It is convenient to assume that the orbits
$x_i$ are constant and hence the neighborhoods $U_i$ are simply small
disjoint balls in $M$. Note that the orbits can be made constant by
composing the flow of $H$ with contractible loops in the group of
Hamiltonian diffeomorphisms; see \cite[Sect.\ 5.1]{Gi:CC}. (This is
not necessary and we can take a small neighborhood of the image of
$x_i$ as $U_i$.  However, in this case we also need to require the
images of $x_i$ to be disjoint.) Assume, furthermore, that no $k$- or
$(k+1)$-periodic orbit of $H$, other than $x_i^k$ or $x_i^{k+1}$,
enters any of the neighborhoods $U_i$.

Let $\tH$ and $K$ be small non-degenerate perturbations of $H$ and
$H^{\nat k}$ such that the regularity conditions for the pair-of-pants
product between $\tH$ and $K$ are satisfied.  (Here we think of $K$ as
a $k$-periodic Hamiltonian, and hence $\tH\nat K$ is
$(k+1)$-periodic.)  Let $u$ be a pair-of-pants curve from a capped
one-periodic orbit $\bx$ of $\tH$ and a capped $k$-periodic orbit
$\by$ of $K$ to a capped $(k+1)$-periodic orbit $\bz$ of $\tH\nat K$.

\begin{Proposition}[Lusternik--Schnirelmann inequality for the
  product]
\label{prop:LS}
Assume that $\tH$ and $K$ are sufficiently $C^\infty$-close to $H$
and, respectively, $H^{\nat k}$. Then there exists $\eps>0$, which
depends only on $H$, $k$ and the neighborhoods $U_i$ but not on the
perturbations $\tH$ and $K$, the curve $u$ or the orbits $\bx$, $\by$
and $\bz$, such that
$$
\CA_{\tH}(\bx)+\CA_{K}(\by)-\CA_{\tH\nat K}(\bz)>\eps
$$
unless $u$ is contained entirely in one of the neighborhoods $U_i$.
\end{Proposition}

\begin{Remark}
  If the image of $u$ is in $U_i$, the orbits $x$, $y$ and $z$ also
  lie in $U_i$. Note also that the condition that the neighborhoods
  $U_i$ are disjoint can be omitted but then the inequality holds when
  $u$ is not contained in a connected component of $\cup U_i$.
\end{Remark}

The proof of the proposition is rather standard and we only spell out
the main idea.

\begin{proof}[Outline of the proof of Proposition \ref{prop:LS}]
  If $u$ does not lie entirely in the neighborhood $U_i$ containing
  $x$, it has to pass through a neighborhood $W$ of the boundary
  $\p U_i=\bar U_i\setminus U_i$.  When the perturbations $\tH$ and
  $K$ are sufficiently close to $H$ and, respectively, $H^{\nat k}$,
  no $k$-periodic orbit of $K$ and $(k+1)$-periodic orbit of
  $\tH\nat K$, other than the orbits which $x^k$ and $x^{k+1}$ split
  into, enters $U_i$. Thus there are no periodic orbits of $H$, $K$
  and $H\nat K$ with periods respectively $1$, $k$ and $k+1$ passing
  through $W$. By \eqref{eq:E}, it is enough to obtain a lower bound
  $\eps>0$ on the energy $E(u)$.  When $E(u)$ is below a certain
  threshold, $\|\p_s u\|_{L^\infty}$ is small and, in fact,
  $\|\p_s u\|_{L^\infty}= o(1)$ as $E(u)\to 0$. Then $\p_t u$ is close
  to the Hamiltonian vector field, since in half-cylindrical parts of
  its domain $\Sigma$ the map $u$ is governed by the corresponding
  Floer equations. Now arguing as in \cite[Sect.\ 1.5]{Sa} or
  \cite[Lemma 19.8]{FO} it is not hard to see that $u$ has to acquire
  a certain amount of energy, \emph{a priori} bounded from below,
  while passing through $W$. As a consequence, we obtain a lower bound
  on $E(u)$.
\end{proof}

\subsection{Spectral invariants and action carriers}
\label{sec:spec}
The theory of Hamiltonian \emph{spectral invariants} was developed in
its present Floer--theoretic form in \cite{Oh,Schwarz}, although the
first versions of the theory go back to \cite{HZ,Vi:gen}. Here we
briefly recall some elements of this theory essential for what
follows, mainly following \cite{GG:gaps}. We refer the reader to
\cite{U2} for a detailed treatment of spectral invariants.

Let $M$ be a closed rational symplectic manifold and let $H$ be a
Hamiltonian on $M$. The \emph{spectral invariant} or \emph{action
  selector} $\s_w$ associated with a class $w\in \HF_*(H)=\HQ_*(M)$
is defined as
$$
\s_w(H)= \inf\{ a\in \R\ssminus \CS(H)\mid w \in \im(i^a)\}
=\inf\{ a\in \R\ssminus \CS(H)\mid j^a\left( w \right)=0\},
$$
where $i^a\colon \HF_*^{(-\infty,\,a)}(H)\to \HF_*(H)$ and
$j^a\colon \HF_*(H)\to\HF_*^{(a,\, \infty)}(H)$ are the natural
``inclusion'' and ``quotient'' maps. It is easy to see that
$\s_w(H)>-\infty$.

The action selector $\s_w$ is a symplectic invariant of $H$ with the
following properties:

\begin{itemize}

             
\item[(AS1)] Continuity: $\s_w$ is Lipschitz in $H$ in the
  $C^0$-topology.

\item[(AS2)] Monotonicity: $\s_w(H)\geq \s_w(K)$ whenever $H\geq K$
  pointwise.

\item[(AS3)] Hamiltonian shift:
$$
\s_w(H+a(t))=\s_w(H)+\int_{S^1}a(t)\,dt,
$$
 where $a\colon S^1\to\R$.

\item[(AS4)] Homotopy invariance: $\s_w(H)=\s_w(K)$ when
  $\varphi_H=\varphi_K$ in the universal covering of the group of
  Hamiltonian diffeomorphisms and $H$ and $K$ have the same mean
  value.

\item[(AS5)] Triangle inequality or sub-additivity:
  $\s_{w_1 * w_2}(H \nat K)\leq\s_{w_1}(H)+\s_{w_2}(K)$.

\item[(AS6)] Spectrality: $\s_w(H)\in \CS(H)$.  More specifically,
  there exists a capped one-periodic orbit $\bx$ of $H$ such that
  $\s_w(H)=\CA_H(\bx)$.

\end{itemize}

The above list of properties of $\s_w$ is far from exhaustive, but it
is more than sufficient for our purposes. Most of these properties are
rather direct consequences of the definition. However, the
sub-additivity, (AS5), relies on \eqref{eq:prod-h}.  It is worth
emphasizing that the rationality assumption plays an important role in
the proofs of the homotopy invariance and spectrality; see
\cite{Oh,Schwarz}. (The latter property also holds in general for
non-degenerate Hamiltonians. This is a non-trivial result; see
\cite{U1}.)

When $H$ is non-degenerate, the action selector can also be evaluated
as
$$
\s_w(H)=\inf_{[\sigma]=w}\CA_H(\sigma),
$$
where we set 
\begin{equation}
\label{eq:cycle-action}
\CA_H(\sigma)=\max\big\{\CA_H(\bx)\,\big|\, \alpha_{\bx} \neq 0\big\}\text{ for }
\sigma=\sum\alpha_{\bx} \bx\in\CF_*(H).
\end{equation}
The infimum here is obviously attained, since $M$ is rational and thus
$\CS(H)$ is closed.  Hence there exists a cycle
$\sigma=\sum\alpha_{\bx} \bx\in\CF_{|w|}(H)$, representing the class $w$,
such that $\s_w(H)=\CA_H(\bx)$ for an orbit $\bx$ entering
$\sigma$. In other words, $\bx$ maximizes the action on $\sigma$ and
the cycle $\sigma$ minimizes the action over all cycles in the
homology class $w$. We call such an orbit $\bx$ a \emph{carrier} of
the action selector. This is a stronger requirement than just that
$\s_w(H)=\CA_H(\bx)$ and $\MUCZ(\bx)=|w|$. When $H$ is possibly
degenerate, a capped one-periodic orbit $\bx$ of $H$ is a carrier of
the action selector if there exists a sequence of $C^2$-small,
non-degenerate perturbations $\tH_i\to H$ such that one of the capped
orbits $\bx$ splits into is a carrier for $\tH_i$. An orbit (without
capping) is said to be a carrier if it turns into one for a suitable
choice of capping.

It is easy to see that a carrier necessarily exists, provided that $M$
is rational.  A carrier is not in general unique, but it becomes
unique when all one-periodic orbits of $H$ have distinct action
values.

As an immediate consequence of the definition of the carrier and
continuity of the action and the mean index, we have
$$
\s_w(H)=\CA_H(\bx)\text{ and } \big|\hMUCZ(\bx)-|w|\big| \leq n,
$$
and the inequality is strict when $x$ is weakly non-degenerate.

Here we will be mainly interested in the spectral invariant associated
with the fundamental class
$w=[M]\in \H_{2n}(M;\F) \subset \HF_*(H)=\HQ_*(M)$ and set for the
sake of brevity
$$
\s(H):=\s_{[M]}(H).
$$
Then $\s(H)=\max H$ when $H$ is autonomous and $C^2$-small, and (AS5)
takes a simpler form:
$$
\s(H \nat
  K)\leq\s(H)+\s(K).
$$
As an immediate consequence of (AS5), we see that the multiplication
map $*[M]$ from \eqref{eq:*M} shifts all spectral invariants by at
most $\s(H)$ upward. A carrier $\bx$ for $\s$ is in some sense
homologically essential. Namely, $\HF_n(\bx)\neq 0$, provided that all
one-periodic orbits of $H$ are isolated; \cite[Lemma 3.2]{GG:nm}. (In
fact, this is true for all spectral invariants $\s_w$, where now
$\HF_{|w|}(\bx)\neq 0$.)

\section{Mean action}
\label{sec:mean}
\subsection{Mean action and the Lusternik--Schnirelmann inequality}
Consider the sequence $\s_k:=\s(H^{\nat k})$. By (AS5), the sequence
$\s_k$ is sub-additive, i.e.,
$$
\s_{k+l}\leq \s_k+\s_l,
$$
and the normalized sequence $\hc_k=\hc_k(H)=\s_k/k$ converges. In
fact, we have a slightly more precise result. Namely,

\begin{equation}
\label{eq:lim}
\hc_\infty:=\lim_{k\to\infty}\hc_k=\inf_k \hc_k.
\end{equation}
This is an easy consequence of the sub-additivity of $\s_k$ and of the
obvious fact that $\hc_\infty>-\infty$; see, e.g., \cite[p.\ 37,
Problem 98]{PS}. In particular,
$$
\hc_k\geq \hc_\infty.
$$
Moreover, under certain natural additional conditions we have a strict
inequality along the lines of Lusternik-Schnirelmann theory as
interpreted in \cite{GG:gaps} and \cite{GG:convex}. This inequality,
Theorem \ref{thm:c-infty0}, plays a key role in the proof of Theorem
\ref{thm:main0}. For the reader's convenience we state the result
again.

\begin{Theorem}[Lusternik--Schnirelmann inequality for mean spectral
  invariants]
\label{thm:c-infty}
Assume that $M$ is rational, all periodic orbits of $H$ are isolated
and none of the orbits is an SDM. Then, for all $k$,
$$
\hc_k>\hc_\infty.
$$
\end{Theorem}

We give a self-contained and detailed proof of the theorem in the next
section.

\begin{Remark}
  The underlying principle behind this theorem is that it is very easy
  for the action to go down in the triangle inequality, (AS5), and the
  equality in (AS5) imposes strong restrictions on the behavior of
  periodic orbits. For instance, one can infer from the equality that
  $w_1$ and $w_2$ admit carriers $\bx$ and $\by$ with common initial
  condition, i.e., $x(0)=y(0)$, provided that all one-periodic orbits
  of $H$ and $K$ are isolated. This can be proved similarly to Theorem
  \ref{thm:c-infty} with suitable modifications and, in fact,
  simplifications.

\end{Remark}

\begin{Remark}
  In Theorem \ref{thm:c-infty} the assumption that the periodic orbits
  are isolated is certainly necessary. (For instance, when $H=\const$,
  we have $\hc_k=\hc_\infty$ for all $k$. There are also ``less
  trivial'' examples.) It is also likely that the condition that none
  of the orbits is an SDM is essential, although we do not have an
  example readily showing this.
\end{Remark}

\begin{Remark}
  In general, the proof of Theorem \ref{thm:c-infty} relies on the
  machinery of virtual cycles (and hence, in particular, we must set
  $\F=\Q$), unless the manifold $M$ is assumed to be weakly
  monotone. In this paper we apply Theorem \ref{thm:c-infty} in the
  setting of Proposition \ref{prop:omega} when $I_\omega=0$, and hence
  the latter requirement is automatically satisfied.
\end{Remark}

Finally, denote by $\widehat{\CS}(H)$ the \emph{normalized or mean
  action spectrum} of $H$, i.e., the union of the increasing sequence
of the nested sets $\CS\big(H^{\nat k}\big)/k$:
$$
\widehat{\CS}(H)=\bigcup_k \frac{1}{k} \CS\big(H^{\nat k}\big).
$$
This set arises naturally in the proofs of some of the Conley
conjecture type results; see \cite{GG:gaps}. Clearly, just because
$\hc_k\to \hc_\infty$ and $\hc_k\in\widehat{\CS}(H)$, we have
\begin{equation}
\label{eq:inf}
\hc_\infty(H)\geq \inf \hat{\CS}(H).
\end{equation}
The invariant $\hc_\infty(H)$ is closely related to Calabi
quasi-morphisms; see \cite{EP,McD,Os,U2}. For us, however,
$\hc_\infty$ is of interest because of its role in the proofs of the
main theorems.

\begin{Remark}
  The connection between $\hc_\infty$ and the Calabi invariant, established
  in \cite{EP}, along with \eqref{eq:lim}, \eqref{eq:inf} and Theorem
  \ref{thm:c-infty} seem to suggest that one might be able to reprove
  and extend to higher dimensions recent results of Hutchings from
  \cite{Hu} by using ``conventional'' symplectic topological
  techniques, not relying on the ECH machinery. However, we have not
  been able to do this.
\end{Remark}

\subsection{Proof of Theorem \ref{thm:c-infty}}
The proof of the theorem is carried out in several steps.

\emph{Step 1: Preparation and action stabilization.}  First, note that
it is sufficient to show that $\hc_l\neq \hc_\infty$ for any $l$; for
$\hc_\infty=\inf \hc_l$ by \eqref{eq:lim}. Arguing by contradiction,
assume that $\hc_l=\hc_\infty$ for some $l$. Then, since the sequence
$\s(H^{\nat k})$ is sub-additive by (AS5) and again by \eqref{eq:lim},
we have $\hc_l=\hc_{2l}=\hc_{3l}=\ldots =\hc_\infty$. Replacing $H$ by
$H^{\nat l}$, we obtain a Hamiltonian for which the sequence $\hc_k$
stabilizes in the first term:
$$
\hc_1=\hc_2=\hc_3=\ldots=\hc_\infty.
$$
Furthermore, replacing $H$ by $H-\hc_\infty$, we can also ensure that
$\hc_\infty=0$. To summarize, we now have
\begin{equation}
\label{eq:c=0}
\s\big(H^{\nat k}\big)=0
\end{equation}
for all $k\in \N$.

It will also be convenient to assume that all one-periodic orbits of
$H$ are constant. This can always be achieved by composing the flow
$\varphi_H^t$ with contractible loops in the group of Hamiltonian
diffeomorphisms; see, e.g., \cite[Sect.\ 5.1]{Gi:CC}.

Finally, we will focus on the case where $N<\infty$, i.e.,
$I_{c_1}\neq 0$. For symplectic CY manifolds, the proof is simpler,
but the wording requires superficial modifications.

\emph{Idea of the proof.} With Step 1 completed, before turning to the
actual proof of the theorem, let us outline the idea of the
argument. To this end, we need to make several simplifying
assumptions. Namely, let us assume that $H$ is (strongly)
non-degenerate, all one-periodic orbits have distinct action, and that
the regularity conditions for the pair-of-paints product are satisfied
for the Hamiltonians $H$ and $H^{\nat k}$ for all $k$.  (Note that the
strong non-degeneracy assumption supersedes the condition that none of
the periodic orbits of $H$ is an SDM.) Let $\bx$ be the unique action
carrier for $\s(H)$. We claim that $\bx^k$ is an action carrier for
$\s\big(H^{\nat k}\big)$. This would be obvious by \eqref{eq:c=0} if
we assumed in addition that all $k$-periodic orbits have
distinct actions since then $\bx^k$ would be the only orbit with 
action equal to $\s\big(H^{\nat k}\big)$.

Let 
$$
\Sigma=\bx+\ldots,
$$
be an action minimizing cycle, where here and below the dots stand for
the terms with action less than zero. Set
$$
C_k:=\Sigma^{ * k}. 
$$
The homology class $[C_k]$ is the fundamental class $[M]$.  Clearly,
$\CA_{H^{\nat k}}(C_k)\leq 0$ and, since $\s\big(H^{\nat k}\big)=0$,
the cycle $C_k$ is also action minimizing and
$$
\CA_{H^{\nat k}}(C_k)=0.
$$
Next, it is not hard to see from Proposition \ref{prop:LS} that the
only term in $C_k$ with zero action is $\beta\bx^k$ for some
$\beta\in \F$:
$$
C_k=\beta\bx^k+\ldots.
$$
By \eqref{eq:c=0}, $\beta\neq 0$ and $\bx^k$ is an action
carrier for $\s\big(H^{\nat k}\big)$.

Finally, we have $\MUCZ(\bx)=n$ and hence, since $x$ is
non-degenerate, $\hMUCZ(\bx)>0$ by \eqref{eq:mean-cz}. Therefore, by
\eqref{eq:mean-hom} and again \eqref{eq:mean-cz},
$\MUCZ(\bx^k)\to\infty$ which is impossible because $\bx^k$ is an
action carrier and thus $\MUCZ(\bx^k)=n$.

None of our simplifying assumptions are satisfied in general, and one
difficulty that arises in the proof is that to meet the regularity
conditions for the pair-of-pants product one has to perturb $H$ and
$H^{\nat k}$ independently. Furthermore, even if regularity could be
achieved by a perturbation of $H$ only, we would not be able to
establish the strict inequality $\hc_k(H)>\hc_\infty(H)$ by passing to
the limit over perturbations.

\emph{Step 2: Orbit stabilization.}  Let us fix a large positive
integer $k_0$ to be specified later. Consider a sequence of
non-degenerate perturbations $H_k$, $k=1,\ldots,k_0$, of $H$. More
specifically, we pick a (strongly) non-degenerate perturbation $H_1$
of $H$ and let $H_k$ with $k=2,\ldots, k_0$ be small perturbations of
$H_1$.  In what follows, we will need to repeatedly require $H_1$ to
be close to $H$ and $H_k$ close to $H_1$, and we will not keep track
of how small these perturbations must actually be.

Denote by $\{\bx_1,\ldots, \bx_r\}$ the collection of all capped
one-periodic orbits of $H$ with $\CA_H(\bx_i)=0$ and
$0\leq \hMUCZ(\bx_i)\leq 2n$.  (Note that when
$\omega\!\mid_{\pi_2(M)}=0$ an orbit $x_i$ can enter this list several
times. Later on we will discard the capped orbits with
$\HF_n(\bx_i)=0$ from this list but at the moment it is convenient to
consider all orbits with zero action.)  Under the perturbation $H_k$
each of the orbits $x_i$ breaks down into non-degenerate orbits
$x_{ij}(k)$.  When $H_k$ is close to $H_1$, each of the orbits
$x_{ij}(k)$ is a small perturbation of $x_{ij}(1)$. In particular,
there is a one-to-one correspondence between the orbits $x_{ij}(1)$
and $x_{ij}(k)$. Hence, in what follows, we denote these orbits by
$x_{ij}$ suppressing $k$ in the notation. This one-to-one
correspondence extends to capped orbits in a natural way with orbits
inheriting the capping from $\bx_i$. Then
$\MUCZ(\bar{x}_{ij}(k))=\MUCZ(\bar{x}_{ij}(1))$. Thus we can just use
the notation $\MUCZ(\bar{x}_{ij})$. Furthermore, $\CA_{H_k}(\bx_{ij})$
is close to $\CA_H(\bx_i)$.

More generally, when all $H_k$ are close to $H_1$, there is a natural
isomorphism between the Floer complexes $\CF_*(H_k)$ for a suitable
choice of an almost complex structure. With this in mind, we can
identify the complexes $\CF_*(H_k)$ with one complex which we denote
by $\CF_*(H_1)$. In particular, $\bar{x}_{ij}$ is the image of
$\bar{x}_{ij}(k)$ under this identification. Note, however, that this
identification preserves the action filtration only up to an error
bounded by the Hofer distance between the Hamiltonians.

Recall that the orbits $x_i$ are assumed to be constant. Fix small
neighborhoods $U_i$ of $x_i$ such that these neighborhoods are
disjoint from each other and from other periodic orbits of $H$ of
period up to $k_0$. Let $\eps_0$ be the minimum of $\eps$ from
Proposition \ref{prop:LS} as $k$ ranges from 1 to $k_0-1$.  Set
$$
\eta=\min\{\delta,\eps_0/2\},
$$
where $\delta$ is half of the distance from $0$ to other points of
$\CS(H)$. In other words,
$$
\delta=\frac{1}{2}\inf\big\{|a| \,\big|\, 0\neq a\in \CS(H)\big\}.
$$

When $H_k$ is sufficiently close to $H$, we have, using the notation
from \eqref{eq:cycle-action},
\begin{equation}
\label{eq:close}
\big|\s(H_k)\big|<\eta\textrm{ and } \big|\CA_{H_k}(\bx_{ij})\big|<\eta,
\end{equation}
and the orbits $\bx_{ij}$ are the only orbits of $H_k$ with action in
the interval $[-\delta,\,\delta]$.

Pick a cycle $\Sigma_k \in \CF_n(H_k)$ over $\F$ representing the
fundamental class $[M]$ and such that
$$
\CA_{H_k}(\Sigma_k)=\s(H_k).
$$
We write
\begin{equation}
\label{eq:Sigma1}
\Sigma_k=\sum \alpha_{ij}\bar{x}_{ij}+\dots,
\end{equation}
where henceforth the dots stand for the terms with action less than
$-\eta$ and in fact less than $-\delta$. The sum extends only over the
capped orbits with index $n$. Note that no orbits with action greater
than $-\eta$ other than $x_{ij}$ can enter $\Sigma_k$. (Otherwise, we
would have $\s(H_k)> \eta$ since $\Sigma_k$ is an action minimizing
cycle.)  When the Hamiltonians $H_k$ are close to $H_1$, these cycles
can be chosen so that the coefficients $\alpha_{ij}$ are independent
of $k$. In other words, we can think of $\Sigma_k$ as one cycle
$\Sigma$ in the complex $\CF_*(H_1)$ obtained by identifying the
complexes $\CF_*(H_k)$.

The capping of the orbit $\bx_{ij}$ comes from the capping of
$\bx_i$. Thus the same orbit $x_{ij}$ may enter \eqref{eq:Sigma1}
several times with different cappings. (This can happen when
$\omega\!\mid_{\pi_2(M)}=0$ and $N\leq n$, and, in particular, $x_i$
can contribute to the list of capped orbits with action zero more than
once.)  To account for this, it is convenient to rewrite
\eqref{eq:Sigma1} as the sum
\begin{equation}
\label{eq:Sigma2}
\Sigma_k=\sum_{i,A}\sum_j \alpha_{ij,A}\cdot({x}_{ij},A)+\dots,
\end{equation}
where $A$ runs over all cappings of $x_{ij}$ or equivalently of $x_i$
and we have set $\alpha_{ij,A}=0$ when $\MUCZ({x}_{ij},A)\neq n$. For
a fixed $i$ we can have only a finite number of different cappings $A$
occurring in this sum with $\alpha_{ij,A}\neq 0$.

Before proceeding with the argument, we need to ``trim'' the cycle
$\Sigma_k$ to guarantee that it involves only homologically essential
orbits $x_i$. Consider the (unordered) collection
\begin{equation}
\label{eq:vec-alpha}
\vec{\alpha}_{i,A}=\{
\alpha_{ij,A}\}.
\end{equation}
We say that $\vec{\alpha}_{i,A}=0$ when $\alpha_{ij,A}=0$ for all $j$.

\begin{Lemma}
\label{lemma:trim} The cycles $\Sigma_k$ can be chosen so that
\begin{equation}
\label{eq:local}
\vec{\alpha}_{i,A}\neq 0\quad \Longrightarrow\quad
\HF_n(x_i,A)\neq 0.
\end{equation}
In other words, constructing the cycle $\Sigma_k$ we can discard from
the collection of the orbits $\bx_i=(x_i,A)$ the orbits with
$\HF_n(\bx_i)=0$.
\end{Lemma}

Note that $\vec{\alpha}_{i,A}\neq 0$ for at least one $i$ for any
choice of $\Sigma_k$; for otherwise we would have $\s(H_k)<-\eta$.

\begin{proof} 
  Due to an isomorphism between the complexes $\CF_*(H_k)$, it is
  enough to do this for one value of $k$. Thus we drop $k$ from the
  notation of the cycle $\Sigma_k$. The cycle $\Sigma$ lies in
  $\CF_*^{(-\infty,\eta)}(H_k)$. Hence it can be projected to
  $\CF_*^{(-\eta,\eta)}(H_k)$. When the Hamiltonian $H_k$ is close to
  $H$, we have the direct sum decomposition of complexes
\begin{equation}
\label{eq:sum}
 \CF_*^{(-\eta,\eta)}(H_k)=\bigoplus_{i,A}\CF_*(x_i,A),
\end{equation}
where the complex $\CF_*(x_i,A)$ is spanned by the capped periodic
orbits $(x_{ij},A)$ which $(x_i,A)$ splits into. (Here we are using
the fact that a Floer trajectory which is not entirely contained in
one of the neighborhoods $U_i$ must have energy bounded from below by
some constant independent of the perturbations $H_k$; see \cite[Sect.\
1.5]{Sa} or \cite[Lemma 19.8]{FO}.) The homology of this complex is
the local homology $\HF_*(x_i,A)$. Let $\Sigma_{i,A}$ be the
projection $\Sigma$ to $\CF_*(x_i,A)$.

Clearly, $\HF_n(x_i,A)\neq 0$ when $[\Sigma_{i,A}]\neq 0$ in
$\HF_*(x_i,A)$. Thus it is sufficient to eliminate the entries
$\alpha_{ij,A}$ for all $i$ such that $[\Sigma_{i,A}]=0$. Pick one
such $i$ and let $\beta$ be a primitive of $\Sigma_{i,A}$ in
$\CF_*(x_i,A)$. We can view $\CF_*(x_i,A)$ as a subspace of
$\CF_*(H_k)$ and thus $\beta\in\CF_*(H_k)$. Consider the cycle
$\Sigma'=\Sigma-\p \beta$, where $\p$ is the differential in the total
complex $\CF_*(H_k)$.  This cycle still represents $[M]$ and, as is
easy to see, $\CA_{H_k}(\Sigma')=\s(H_k)$. Furthermore,
$\vec{\alpha}_{i,A}=0$ for $\Sigma'$, while other groups
$\vec{\alpha}_{i',A'}$, where $i'\neq i$ or $A'\neq A$, remain the
same as for $\Sigma$. (However, the ``lower order terms'', i.e., the
terms with action below $-\eta$, can be effected by this change.)
Applying this procedure to all $i$ and $A$ with $[\Sigma_{i,A}]=0$, we
obtain a new cycle, which we still denote by $\Sigma$ or $\Sigma_k$,
satisfying~\eqref{eq:local}.
\end{proof}

Note that $\hMUCZ(x_i,A)> 0$ for every orbit $(x_i,A)$ with
$\HF_n(x_i,A)\neq 0$ and hence, by \eqref{eq:local}, for every orbit
with $\vec{\alpha}_{i,A}\neq 0$.  Indeed, $\hMUCZ(x_i,A)\geq 0$ by
\eqref{eq:mean-cz}. If we had $\hMUCZ(x_i,A)=0$, the orbit $(x_i,A)$
would be an SDM, which would contradict the conditions of the
theorem. Thus $\hMUCZ(x_i,A)> 0$ for all orbits $(x_i,A)$ that enter
the cycle $\Sigma_k$ with $\vec{\alpha}_{i,A}\neq 0$.

Next, consider the collection of capped $k$-periodic orbits of $H$
with zero action. Among these are the iterated orbits $\bx_i^k$ where
$\HF_n(\bx_i)\neq 0$, but in general this collection may include some
other orbits. Moreover, when $I_\omega=0$, an orbit $x_i^k$ may occur
in this collection with a capping different from that of $\bx_i^k$.

Denote by $y_i^{(k)}$ the $k$-periodic orbits of $H$. For a generic
choice of the perturbations $H_k$, the Hamiltonian
$$
F_k=H_1\nat\cdots \nat H_k
$$
is a small perturbation of the non-degenerate Hamiltonian
$H_1^{\nat k}$ and hence of $H^{\nat k}$.  (At this point it is
essential that the range of $k$ is fixed and finite.) Under this
perturbation, the orbits $y_i^{(k)}$ split into non-degenerate orbits
which we denote by $y_{ij}^{(k)}$. For instance, when
$y_i^{(k)}=x_i^k$, these are the orbits $x_{ij}^k$ or, to be more
precise, small perturbations $x_{ij}^{(k)}$ of these orbits, and
perhaps some other orbits. The capping of $y_i^{(k)}$ gives rise to
cappings of the orbits it splits into. In particular, when
$\by_i^{(k)}$ has zero action, we denote the resulting capped orbits
by $\by_{ij}^{(k)}$.  This collection of orbits of $F_k$ includes the
orbits $\bx_{ij}^{(k)}$ when $\bx_i$ has zero action, possibly the
orbits $x_{ij}^{(k)}$ with other cappings, and finally some other
orbits.

When the Hamiltonians $H_k$ are close to $H_1$, there is again a
natural isomorphism between the Floer complexes $\CF_*(F_k)$ and
$\CF_*(H_1^{\nat k})$. Identifying these complexes, we can view the
orbits $\by_{ij}^{(k)}$ as elements of $\CF_*(H_1^{\nat k})$.

Similarly to \eqref{eq:close}, we can ensure that 
$$ 
|\s(F_k)|<\eta\textrm{ and }
\big|\CA_{F_k}\big(\by_{ij}^{(k)}\big)\big|<\eta,
$$
and $\by_{ij}^{(k)}$ are the only orbits of $F_k$ with action in
$[-\delta,\delta]$ by taking $H_k$ close to $H$. (At this point again
it is essential that $k$ is taken within a finite range
$1,\ldots,k_0$.)

Consider now the cycle
$$
C_{k}=\Sigma_1 *\cdots * \Sigma_k\in \CF_*(F_k)
$$
representing the fundamental class $[M]$ in $\HF_*(F_k)$. Since the
pair-of-pants product is not associative on the level of complexes,
the placing of parentheses matters here. Thus, to be more precise, the
cycle $C_k$ is defined inductively by
$$
C_{k+1}=C_{k}*\Sigma_{k+1}
$$
with $C_1=\Sigma_1$. Alternatively, we can view the multiplication on
the right by $\Sigma_{k+1}$ as a map of complexes
$$
\Phi_{k}\colon \CF_*(F_{k})\to\CF_*(F_{k+1})
$$
and
$$
C_{k+1}=\Phi_{k}\circ\ldots\circ \Phi_1(\Sigma_1).
$$

It is easy to see that 
$$
\big|\CA_{F_k}(C_k)\big|<\eta. 
$$
Indeed, $\CA_{H_k}(\Sigma_k)=\s(H_k)$ and
$$
\CA_{F_k}(C_k) \le 
\sum_{j=1}^k \CA_{H_j}(\Sigma_j)
=
\sum_{j=1}^k \s(H_j).
$$
When all $H_k$ are sufficiently close to $H$, the spectral invariants
$\s(H_k)$ are close to $\s(H)=0$. In particular, we can ensure that
$\CA_{F_k}(C_k)<\eta$.  Furthermore, $\CA_{F_k}(C_k)>-\eta$ because
$\s(F_k)>-\eta$. In other words, $C_k\in \CF_n^{(-\infty,\eta)}(F_k)$.

Let us denote by $C_k'$ the natural projection of $C_k$ to the
quotient complex $\CF_n^{(-\eta,\eta)}(F_k)$. In other words, $C_k'$
is the leading term in the expression
\begin{equation}
\label{eq:C_k}
C_k=\sum \beta_{ij}^{k}\by_{ij}^{(k)}+\dots,
\end{equation}
where, as before, the dots stand for the terms with action less than
$-\eta$, and the sum extends only over the capped orbits with index
$n$. Similarly to \eqref{eq:sum}, when the Hamiltonians $H_k$ are
sufficiently close to $H$, the complex $\CF_n^{(-\eta,\eta)}(F_k)$
decomposes into a direct sum of complexes $\CF_*\big(y_i^{(k)}\big)$
formed by the orbits which $y_i^{(k)}$ with all possible cappings
breaks into:
\begin{equation}
\label{eq:sum2}
\CF_*^{(-\eta,\eta)}(F_k)=\bigoplus_{i}\CF_*\big(y_i^{(k)}\big).
\end{equation}
We denote by $C_{k,i}$ the projection of $C_k'$ to
$\CF_*\big(y_i^{(k)}\big)$.

By Proposition \ref{prop:LS} and the choice of $\eta$,
$$
\by_{i'j'}^{(k)}*\bx_{ij}=\ldots\textrm{ when } y_{i'}^{(k)}\neq
x_{i}^k ,
$$
where again the dots denote the terms with action less than $-\eta$.
Hence the operators $\Phi_k$ block-diagonalize with respect to the
decomposition \eqref{eq:sum2}:
$$
\Phi_k=\bigoplus_i \Phi_{k,i}.
$$ 
Morerover, $\Phi_{k,i}=0$ when $y_i^{(k)}$ is not one of the orbits
$x_i^k$. In other words, the leading term $C_k'$ in \eqref{eq:C_k}
involves only the orbits $x_{ij}^{(k)}$ with some cappings.

The key result of this step is the following.

\begin{Lemma}
\label{lemma:carrier-stab}
For some $i$ independent of $k\leq k_0$ we have $C_{k,i}\neq 0$, i.e.,
there exists a sequence of capped orbits $\big(x_{ij}^{(k)},B_k\big)$,
indexed by $k$, with $B_k$ and $j$ possibly depending on $k$, entering
the cycles $C_k$ for $1\leq k\leq k_0$ with non-zero coefficients.
\end{Lemma}

\begin{proof} 
  Since the operators $\Phi_k$ block-diagonalize, we have
$$
C_{k,i}=\Phi_{k-1,i}(C_{k-1,i}).
$$
As a consequence,
$$
C_{k,i}= 0\quad \Longrightarrow\quad
C_{k+1,i}= 0,\,\ldots,\, C_{k_0,i}= 0.
$$
Clearly, $C_k\neq 0$ since $[C_k]=[M]$ and hence $C_{k,i}\neq 0$ for
all $k\leq k_0$ for some $i$ independent of $k$.
\end{proof}

In the setting of Lemma \ref{lemma:carrier-stab}, $C_{1,i}$ is the sum
of the cycles $\Sigma_{i,A}$ for all suitable cappings $A$, and hence
$\Sigma_{i,A}\neq 0$ for some $A$. Then for $\bx_i=(x_i,A)$ we also
have
\begin{equation}
\label{eq:x_i}
\HF_n(\bx_i)\neq 0\textrm{ and }\hMUCZ(\bx_i)>0
\end{equation}
by Lemma \ref{lemma:trim}. Here the second inequality is a consequence
of the first one and the fact that $H$ does not have SDM orbits.

\begin{Remark}
  The role of Lemma \ref{lemma:trim} in this argument is to ensure
  that \eqref{eq:x_i} is satisfied for the orbit $(x_i,A)$. There is a
  different way to do this. First, note that for $i$ as in Lemma
  \ref{lemma:carrier-stab}, some orbit $\bx_{ij}^{(1)}=\bx_{ij}$ is a
  carrier for $H_1$. (In contrast, the orbits $\bx_{ij}^{(k)}$ with
  $k\geq 2$ may fail to be action carriers, although they are carriers
  up to a small error $\eta$. The reason is that the cycle $C_k$ is
  action minimizing only up to $\eta$.) Taking a sequence of
  Hamiltonians $H_1\to H$ and also the Hamiltonians $H_k\to H$,
  choosing the cycles $\Sigma_k$ for them, and applying Lemma
  \ref{lemma:carrier-stab}, we obtain a sequence of carriers
  $\bx_{ij}$ for $H_{1}$. Passing to a subsequence, we can guarantee
  that $i$ and the inherited capping $A$ of $x_i$ are independent of
  $k$. Then $(x_i,A)$ is an action carrier and thus
  $\HF_n(x_i,A)\neq 0$; \cite[Lemma 3.2]{GG:nm}.
\end{Remark}

\emph{Step 3: Index growth.}  First, let us specify the value of
$k_0$.  Consider all capped one-periodic orbits $(x_i,A)$ entering the
cycle $\Sigma$ with $\vec{\alpha}_{i,A}\neq 0$ or more generally all
capped one-periodic orbits with $\HF_n(x_i,A)\neq 0$. As has been
pointed out in Step 2, $\hMUCZ(x_i,A)>0$ for each such
orbit. Therefore, since $\PP_1(H)$ is finite,
$$
\Delta:=\min_{(x_i,A)} \hMUCZ(x_i, A)>0.
$$
We set
\begin{equation}
\label{eq:k0}
k_0:=\left\lceil \frac{2n+2}{\Delta}\right\rceil.
\end{equation}

Next, consider the orbits $x_i$ and $\bx_{ij}^{(1)}$ from Lemma
\ref{lemma:carrier-stab}. For the sake of brevity, set
$$
z:=x_i\textrm{, } z_j:=x_{ij}\textrm{ and
}\bz^{(k)}_j:=\big(x_{ij}^{(k)},B_k\big).
$$
The orbit $z^{(k)}_j$, with capping ignored, is a small perturbation
of $z_j^k$ and we may simply identify these orbits.  Since $z$ is a
constant orbit (see Step 1), we can take the trivial (i.e., constant)
capping of $z$ as a reference.  The orbits $z_j^k$ need not be
constant. However, each of these orbits is contained in a small ball
$U$ centered at $z$ and we can fix a capping contained in $U$ as a
reference capping for $z_j^k$. With this in mind, the collections of
all cappings of $z$ or $z_j^k$ can be identified with the group
$\Gamma=\pi_{2}(M)/\ker I_{c_1}\cap \ker I_\omega$. In particular, we
have a one-to-one correspondence between the cappings of these orbits.

When the Hamiltonians $H_k$ are sufficiently close to $H$, every
pair-of-pants curve connecting $\bz_{j'}^{(1)}$ and
$\bz_{j''}^{(k-1)}$ to $\bz_j^{(k)}$ must be trivial, i.e., contained
in $U$. For, otherwise, we would have
$\CA_{H^{\nat k}}\big(\bz_j^{(k)}\big)<-\eta$ by Proposition
\ref{prop:LS} and the assumption that the perturbations $H_k$ are
sufficiently close to $H$.  Then, as follows from the definition of
the cycle $C_k$,
$$
\bz^{(k)}_j=\big(z_j^k, A_1+\cdots+A_k\big),
$$
where each $A_l$ is one of the cappings $A$ of $z$ with
$\HF_n(z,A_l)\neq 0$. 

Let $A_0$ be the capping of $z$ such that $\Delta_0:=\hMUCZ(z,A_0)$ is
the smallest possible value of $\hMUCZ(z,A)$ when $\HF_n(z,A)\neq
0$. Clearly, $\Delta_0\geq \Delta>0$ and
$$
I_{c_1}(A_l-A_0)\geq 0.
$$
Hence,
$$
n=\MUCZ\big(\bz_j^{(k)}\big)=\MUCZ\big(z_j^k,kA_0\big)+\sum_l
I_{c_1}(A_l-A_0)\geq \MUCZ\big(z_j^k,kA_0\big)
$$
for all $k\leq k_0$. On the other hand, $\hMUCZ(z_j,A)$ can be made
arbitrarily close to $\hMUCZ(z,A)$ uniformly in $A$ by taking the
Hamiltonians $H_k$ close to $H$. Therefore,
$$
\MUCZ\big(z_j^k,kA_0\big) \geq k\Delta_0-n-1 \geq k\Delta-n-1,
$$
where we subtracted $n+1$ rather than $n$ to account for the
discrepancy between $\hMUCZ(z_j,A_0)$ and $\hMUCZ(z,A_0)$.  Setting
$k=k_0$ and using \eqref{eq:k0}, we arrive at a contradiction:
$$
n=\MUCZ\big(\bz_j^{(k)}\big)\geq k_0\Delta-n-1 \geq n+1.
$$
$\hfill\square$

\section{Proof of the main theorem}
\label{sec:pf}
For the reader's convenience we begin with restating the main theorem
(Theorem \ref{thm:main0}) of the paper.

\begin{Theorem}
\label{thm:main2}
Assume that a closed symplectic manifold $M$ admits a Hamiltonian
diffeomorphism $\varphi_H$ with finitely many periodic orbits. Then
there exists $A\in\pi_2(M)$ such that $\omega(A)>0$ and
$\left<c_1(TM),A\right>>0$.
\end{Theorem}

This result is a formal consequence of already known cases of the
Conley conjecture and the following proposition.

\begin{Proposition}
\label{prop:omega}
Assume that a closed symplectic manifold $M$ admits a Hamiltonian
diffeomorphism $\varphi_H$ with finitely many periodic orbits. Then
$\omega|_{\pi_2(M)}\neq 0$, i.e., $I_\omega\neq 0$.
\end{Proposition}

\begin{proof}[Proof of Theorem \ref{thm:main2}]
  Set $V_\R:=\pi_2(M)\otimes\R$. Clearly, the homomorphisms $I_\omega$
  and $I_{c_1}$ extend to $V_\R$. Since the Conley conjecture is known
  to hold for symplectic Calabi-Yau manifolds (see
  \cite{GG:gaps,He:irr}), we have $I_{c_1}\neq 0$ on $V_\R$ under the
  assumptions of the theorem. Likewise, $I_\omega\neq 0$ by
  Proposition \ref{prop:omega}.

  If $\ker I_{c_1}=\ker I_\omega$ on $V_\R$, the manifold $M$ is
  either negative monotone or strictly positive monotone. The former
  case is ruled out since the Conley conjecture holds for negative
  monotone manifolds as is shown in \cite{GG:nm}. In the latter case,
  there is $A\in\pi_2(M)$ such that $\omega(A)>0$ and
  $\left<c_1(TM),A\right>>0$ and the proof is finished.

  When $\ker I_{c_1}\neq\ker I_\omega$ on $V_\R$, there exists
  $A'\in V_\R$ such that $I_{c_1}(A')<0$ and $I_\omega(A')<0$. The
  space $V_\Q:=\pi_2(M)\otimes \Q$ is dense in $V_\R$ and hence there
  is $A''\in V_\Q$ with similar properties. Finally, $\omega(A)>0$ and
  $\left<c_1(TM),A\right>>0$ where $A=m A''\in \pi_2(M)$ for a
  suitably chosen $m\in \Z$.
\end{proof}

\begin{proof}[Proof of Proposition \ref{prop:omega}]
  Arguing by contradiction, assume that
  $\omega\!\mid_{\pi_2(M)}=0$. Note that, as a consequence, the Floer
  homology of $H$ is defined without virtual cycles and $(M,\omega)$
  is automatically rational with $\lambda_0=\infty$. None of the
  periodic orbits of $H$ is an SDM, for otherwise $H$ would have
  infinitely many periodic orbits. Furthermore, $\CS(H)$ is finite and
  hence the sequence $\hc_k$ necessarily stabilizes since it is
  converging, which is impossible by Theorem \ref{thm:c-infty}.
\end{proof}

\begin{Remark}
  The proof of Proposition \ref{prop:omega} essentially comprises two
  cases: the ``non-degenerate'' case and the ``degenerate'', i.e., SDM
  case. Interestingly, this structure is common to all symplectic
  topological proofs of the Conley conjecture type results, even
  though the arguments in the non-degenerate case are quite
  different. We also note that Theorem \ref{thm:c-infty}, again in
  conjunction with the SDM case, can also be used to prove the Conley
  conjecture for negative monotone symplectic manifolds, thus
  by-passing a reference to the results from \cite{CGG,GG:nm} in the
  proof of Theorem~\ref{thm:main2}.
\end{Remark}

\begin{Remark}
  Theorem \ref{thm:c-infty} is closely related to and can be proved
  using \cite[Prop.\ 5.3]{GG:gap} asserting that the pair-of-pants
  product in the local Floer homology of an isolated non-SDM orbit is
  nilpotent. However, such an argument, when detailed, would not be
  much different from or much simpler than the self-contained proof in
  the previous section. On the other hand, one can directly prove
  Theorem \ref{thm:main2} by establishing Proposition \ref{prop:omega}
  as a consequence of \cite[Prop.\ 5.3]{GG:gap} through purely
  algebraic means. The essence of the argument is that if we had
  $\omega$ aspherical and simultaneously $\dot{\PP}(H)$ finite, the
  algebra $\bigoplus_{k\geq 1}\HF_*\big(H^{\nat k}\big)$ would
  necessarily be nilpotent, which is, of course, impossible by
  \eqref{eq:*M}.
\end{Remark}

\section{Perfect Hamiltonians and generic existence}
\label{sec:further}
The notion of the augmented action and the asymptotic spectral
invariant $\hc_\infty$ naturally arise in the study of Hamiltonians
with finitely many periodic orbits and, more specifically, perfect
Hamiltonians. Recall that a Hamiltonian $H$ is said to be
\emph{perfect} if $\dot{\PP}(H)$ is finite and
$\PP_1(H)=\dot{\PP}(H)$, i.e., when $H$ has only finitely many simple
periodic orbits and every such orbit is one-periodic. The latter
condition is automatically satisfied in all known examples of
Hamiltonians with finitely many periodic orbits. Furthermore, all
known perfect Hamiltonians are non-degenerate. Clearly, a suitable
iteration of a Hamiltonian with finitely many simple periodic orbits
is perfect.

One class of examples of perfect Hamiltonians is given by generic
elements in Hamiltonian circle or torus actions with isolated fixed
points. There are, however, other examples, e.g., pseudo-rotations,
with extremely non-trivial dynamics; see, e.g., \cite{AK,FK}.

To see the connection between perfect Hamiltonians and asymptotic
spectral invariants, note first that the actual value of $\hc_\infty$
is not completely trivial to calculate directly by definition even for
such a Hamiltonian as a quadratic form on $\CP^n$. However, it has a
simple interpretation as the so-called augmented action. Namely,
assume that $M$ is strictly monotone (or negative monotone) with
monotonicity constant $\lambda \neq 0$; see Section
\ref{sec:conv}. For $x\in\PP_1(H)$, the \emph{augmented action} of $x$
is
$$
\tA_H(x)=\CA_H(\bx)-\frac{\lambda}{2}\hMUCZ(\bx),
$$
where on the right hand side $x$ is equipped with an arbitrary
capping. By \eqref{eq:omega-c1}, the left hand side is well defined,
i.e., independent of the capping. Augmented action was introduced in
\cite{GG:gaps} and then applied in the circle of questions related to
the Conley conjecture in \cite{CGG,GG:nc}. Under certain conditions it
behaves similarly to the ordinary action while being
capping--independent. In particular, the augmented action is
homogeneous with respect to iterations.

\begin{Theorem}
\label{thm:aa-c_infty}
Assume that $M$ is strictly positive monotone and $H$ is perfect. Let
$x$ be a one-periodic orbit whose iterations occur infinitely many
times as carriers for $\s\big(H^{\nat k}\big)$. Then
\begin{equation}
\label{eq:aa-c_infty}
\hc_\infty=\tA_H(x).
\end{equation}
\end{Theorem}

A similar result holds when $M$ is negative monotone, but then the
assertion is void; for negative monotone symplectic manifolds admit no
Hamiltonians with finitely many periodic orbits; \cite{GG:nm}.  Note
also that an orbit $x$ satisfying the hypothesis of Theorem
\ref{thm:aa-c_infty} always exists since $\PP_1(H)=\dPP(H)$ is finite.
The theorem has an analog for Reeb flows; see \cite[Sect.\
6.1.2]{GG:convex}.

\begin{proof}
  Let $\bx_k$ be a carrier for $\s(H^{\nat k})$. To prove
  \eqref{eq:aa-c_infty}, it is enough to show that
$$
\hc_\infty=\lim_{k\to\infty} \frac{1}{k}\tA_{H^{\nat k}}(x_k)
$$
since the augmented action is homogeneous. This readily follows from
that
$$
\big|\tA_{H^{\nat k}}(x_k)-\CA_{H^{\nat k}}(\bx_k)\big|\leq
\frac{\lambda}{2}\cdot 2n=\lambda n,
$$
which is, in turn, an immediate consequence of the fact that
$\HF_n(\bx_k)\neq 0$, and hence $0\leq \hMUCZ(\bx_k)\leq 2n$.
\end{proof}

\begin{Example}
\label{ex:CPn}
Consider a quadratic Hamiltonian
$H(z)=\pi\big(\lambda_0|z_0|^2+\cdots+\lambda_n|z_n|^2\big)$ on
$\CP^n$, where the coefficients $\lambda_0,\ldots,\lambda_n$ are
linearly independent over $\Q$. (Here, we identify $\CP^n$ with the
quotient of the unit sphere in $\C^{n+1}$ and hence $\sum|z_i|^2=1$.)
The Hamiltonian $H$ is perfect and has exactly $n+1$ fixed points, the
coordinate axes. A simple calculation shows that their augmented
actions are equal to $\pi\sum \lambda_i/(n+1)$, cf.\ \cite[Exam.\
1.2]{CGG}. Hence $\hc_\infty=\pi\sum \lambda_i/(n+1)$.
\end{Example}

Now we are in a position to prove Theorem \ref{thm:gap-bound0} showing
that the gap between $\s_{[M]}\big(H^{\nat k}\big)$ and
$\s_{[\pt]}\big(H^{\nat k}\big)$, where $[\pt]\in \HF_{-n}(H)$ is the
homology class of a point, is eventually \emph{a priori} bounded for
perfect Hamiltonians. Let us state the theorem again.

\begin{Theorem}[Action Gap]
\label{thm:gap-bound}
Assume that $H$ is perfect and $M$ is positive monotone with
monotonicity constant $\lambda>0$. Then
\begin{equation}
\label{eq:gap-bound}
\s_{[M]}\big(H^{\nat k}\big)-\s_{[\pt]}\big(H^{\nat k}\big)\leq 2\lambda n,
\end{equation}
for all but possibly a finite number of iterations $k\in\N$.
\end{Theorem}

In this theorem and in the rest of the section, the choice of the
coefficient ring $\F$ is immaterial and the ring is suppressed in the
notation. In what follows, for the sake of brevity, we also set
$$
\s^+_k:=\s_{[M]}\big(H^{\nat k}\big)\textrm{ and }
\s_k^-:=\s_{[\pt]}\big(H^{\nat k}\big)
$$
and
$$
\hc_\infty^\pm=\lim_{k\to\infty}\s_k^\pm/k.
$$
Thus, in particular, $\s^+_k=\s_k$ and $\hc^+_k=\hc_k$ in the notation
of the previous sections. The difference $\s^+-\s^-$ has played a
prominent role in some aspects of symplectic topology and is sometimes
referred to as the $\gamma$-norm or $\gamma$-metric. It was introduced
in \cite{Vi:gen, HZ} and further studied in, e.g., \cite{Oh,
  Schwarz}. It is easy to see from (AS5) in Section \ref{sec:spec}
that
$$
\s^+_k\geq \s^-_k
$$
and, moreover, the inequality is strict unless $\varphi_H=\id$. The
latter assertion is non-trivial and usually proved by comparing
$\s^+-\s^-$ with the capacity of a small displaced ball; see, e.g.,
\cite{Schwarz}. (One can also argue, when $M$ is monotone, as in the
proof of \cite[Prop.\ 6.2]{GG:gaps}.)

Upper bounds on the difference between spectral invariants as in
\eqref{eq:gap-bound} usually result from non-vanishing of certain GW
invariants or relations in the quantum product; see, e.g., \cite{EP,
  GG:gaps}. Thus Theorem \ref{thm:gap-bound} provides further evidence
supporting the Chance--McDuff conjecture discussed in the
introduction.

\begin{proof}[Proof of Theorem \ref{thm:gap-bound}]  
  By rescaling $\omega$, we may assume that $\lambda=2$, i.e.,
  $I_\omega=I_{c_1}$, and then \eqref{eq:gap-bound} turns into
\begin{equation}
\label{eq:gap-bound2}
\s_k^+-\s_k^-\leq 4 n.
\end{equation} Furthermore, it is enough to show that every 
sequence $k_i\to\infty$ contains an infinite subsequence for which
\eqref{eq:gap-bound2} holds.

Thus let $k_i\to\infty$. Since $H$ is perfect, there is a one-periodic
orbit $x$ and an infinite subsequence, which we still denote by $k_i$,
such that the iterations $x^{k_i}$ with some cappings are action
carriers for $\s^+_{k_i}$. Next, in a similar vein, there is a
one-periodic orbit $y$ and again an infinite subsequence $k_{i_j}$,
such that the iterations $y^{k_{i_j}}$ with some cappings are action
carriers for $\s^-_{k_{i_j}}$. Let us denote the sequence $k_{i_j}$ by
$k_i$ again, write $k=k_i$, and let $\bx_{k}$ and, respectively,
$\by_{k}$ be the resulting capped periodic orbits.

We have
$$
\s_k^+-\s_k^-=\CA_{H^{\nat k}}(\bx_k)-\CA_{H^{\nat k}}(\by_k).
$$
By Theorem \ref{thm:aa-c_infty},
$$
\CA_{H^{\nat k}}(\bx_k)=\tA_{H^{\nat
    k}}(\bx_k)+\hMUCZ(\bx_k)=\hc_\infty^+ + \hMUCZ(\bx_k).
$$
To evaluate the second term, recall that by Poincar\'e duality
$$
\s_{[\pt]}(H)=-\s_{[M]}\big(H^{-1}\big),
$$
where $H^{-1}$ is the Hamiltonian generating the inverse
time-dependent flow $\varphi_H^{-1}\varphi_H^{1-t}$. Therefore,
$$
\hc_\infty^+\big(H^{-1}\big)=-\hc^-_\infty(H)=-\hc^-_\infty.
$$
Furthermore, let us denote by $\by_{-k}$ the orbit $y_k$ traversed in
the opposite direction and equipped with the capping obtained from
$\by_k$ by reversing the orientation. Then, setting
$H^{-\nat k}:=\big(H^{-1}\big)^{\nat k}$, we have
\begin{align*}
\CA_{H^{\nat k}}(\by_k) &= -\CA_{H^{-\nat k}}(\by_{-k})\\
&= -\big[\tA_{H^{-\nat k}}(\by_{-k})+\hMUCZ(\by_{-k})\big]\\
&= -\big[\hc_\infty^+\big(H^{-1}\big)-\hMUCZ(\by_{k})\big]\\
&= -\big[-\hc_\infty^--\hMUCZ(\by_{k})\big]\\
&= \hc_\infty^-+\hMUCZ(\by_{k}).
\end{align*}

Therefore,
\begin{equation}
\label{eq:ss}
\s_k^+-\s_k^-=\hc_\infty^+-\hc_\infty^- + \hMUCZ(\bx_k)-\hMUCZ(\by_{k}).
\end{equation}

Since $\bx_k$ and $\by_k$ are action carriers for $\s_k^\pm$, we have
$$
\big|\hMUCZ(\bx_k)-n\big|\leq n   \textrm{ and }
\big|\hMUCZ(\by_{k})+n\big|\leq n.
$$ 
Furthermore, recall that as is well-known
$$
\s_k^+-\s_k^-\leq \big\| H^{\nat k}\big\|,
$$
where $\big\| H^{\nat k}\big\|$ stands for the Hofer norm of
$H^{\nat k}$, and thus
$$
\hc_\infty^+-\hc_\infty^-\leq \| H\|.
$$
As a consequence, by
\eqref{eq:ss},
$$
0\leq \s_k^+-\s_k^-\leq \|H\|+4n.
$$
Dividing by $k$ and passing to the limit as $k\to\infty$ in the
sequence $k_i$, we conclude that
$$
\hc_\infty^+=\hc_\infty^-.
$$
Returning to \eqref{eq:ss}, we now arrive at
\begin{equation}
\label{eq:bound3}
\s_k^+-\s_k^-=\hMUCZ(\bx_k)-\hMUCZ(\by_{k})\leq 4n,
\end{equation}
proving \eqref{eq:gap-bound2}.
\end{proof}

\begin{Remark}
  We conjecture that in the setting of Theorem \ref{thm:gap-bound},
$$
\s^+_{k_i}-\s^-_{k_i}\to 0
$$
for some sequence $k_i\to\infty$. A direct calculation shows that this
is indeed true for rotations of $S^2$ or quadratic Hamiltonians on
$\CP^n$. This is not entirely obvious; for while it is easy to see
that some action gap goes to zero for a subsequence, it is not
immediately clear that this is so for the specific action gap from
\eqref{eq:gap-bound}.

To further elaborate on the conjecture note that when $H$ is
non-degenerate the first part of \eqref{eq:bound3} can be re-written
as
\begin{equation}
\label{eq:bound4}
\s_k^+-\s_k^-=2n+\big(\hMUCZ(\bx_k)-\MUCZ(\bx_k)\big)+\big(\MUCZ(\by_k)-\hMUCZ(\by_k)\big).
\end{equation}
The last two terms on the right hand side are independent of the
cappings of $x^k$ and $y^k$.  It is easy to see that the right hand
side of \eqref{eq:bound4} is bounded away from 0 by at least 1 unless
both $x$ and $y$ are elliptic. When this is the case, it looks very
plausible that indeed $\s_{k_i}^+-\s_{k_i}^-\to 0$ for some
subsequence. (This fact is non-obvious and closely related to the
results in \cite{DLW} and \cite{GG:convex}.)  However, there seems to
be no way of ensuring that this subsequence has an infinite overlap
with the subsequence for which both $x^k$ and $y^k$ are action
carriers.
\end{Remark}

As an immediate application of Theorem \ref{thm:gap-bound}, we obtain
a similar upper bound for other homology classes. For $w\in \H_*(M)$
set
$$
\s^w_k:=\s_w\big( H^{\nat k}\big)\textrm{ and }\,
\hc_\infty^w:=\lim_{k\to\infty}\s^w_k/k. 
$$
Clearly, $\s^-_k\leq \s^w_k\leq \s^+_k$ when $w\neq 0$. The following
result readily follows from Theorem \ref{thm:gap-bound}.

\begin{Corollary}
\label{cor:gap-bound}
Assume that $H$ is a perfect Hamiltonian on a positive monotone
symplectic manifold $M$ with monotonicity constant $\lambda>0$. Then
for any two non-zero classes $w_0$ and $w_1$ in $\H_*(M)$ we have
$$
\big|\s^{w_1}_k-\s^{w_0}_k\big|\leq 2\lambda n\textrm{ and }\,
\hc^{w_1}_\infty=\hc^{w_0}_\infty=\hc_\infty,
$$
where the first inequality holds for all but possibly a finite number
of $k\in\N$.
\end{Corollary}

One important feature of perfect Hamiltonians is the action-index
resonance relations, e.g., the fact that certain periodic orbits have
the same augmented action; see \cite{CGG} and \cite[Cor.\ 1.11 and
Thm.\ 1.12]{GG:gaps}. Theorems \ref{thm:aa-c_infty} and
\ref{thm:gap-bound} enable us to refine these results identifying the
common augmented action value for these orbits with $\hc_\infty$. We
start by stating a general result along these lines.

\begin{Corollary}[Resonance Relations, I]
\label{cor:res1}
Assume that $M^{2n}$ is a strictly positive monotone closed symplectic
manifold such that $N\geq 2$ or $\dim\H_*(M)>n+1$. Let $H$ be a
perfect non-degenerate Hamiltonian on $M$. Then $H$ has two
geometrically distinct one-periodic orbits $x$ and $y$ with
\begin{equation}
\label{eq:res1}
\tA_H(x)=\tA_H(y)=\hc_\infty.
\end{equation}
\end{Corollary}

\begin{proof}
  By our assumptions on $M$, there exists two distinct homology
  classes with $0\leq |w_1|-|w_0|<2N$. (If $N>1$ we can take two
  classes with $|w_1|-|w_0|=2$ and if $N=1$ but $\dim \H_*(M)>n+1$ we
  can take two classes with $|w_1|=|w_0|$.) Let $\bx$ and $\by$ be
  carriers for $\s^{w_1}$ and $\s^{w_0}$. Then the orbits $x$ and $y$
  are geometrically distinct since
$$
\big|\MUCZ(\bx)-\MUCZ(\by)\big|=|w_1|-|w_0|<2N.
$$ 
Now it remains to apply Corollary \ref{cor:gap-bound} and Theorem
\ref{thm:aa-c_infty}.
\end{proof}

This corollary is a (partial) refinement of \cite[Cor.\
1.11]{GG:gaps}. An essential feature of this result is that we make no
assumption about the structure of the quantum product in
$\HQ_*(M)$. Also note that in the setting of the proof of the
corollary, \eqref{eq:bound4} takes the form
$$
\s_k^+-\s_k^-=|w_1|-|w_0|+\big(\hMUCZ(\bx_k)-\MUCZ(\bx_k)\big)+\big(\MUCZ(\by_k)-\hMUCZ(\by_k)\big),
$$
where $\bx_k$ and $\by_k$ are carriers for $\s_k^{w_1}$ and,
respectively, $\s_k^{w_0}$.  As in \cite{GG:generic}, we can infer
from Corollary \ref{cor:res1} the generic existence of infinitely many
periodic orbits.

\begin{Corollary}[Generic Existence]
\label{cor:gen}
Assume that $M^{2n}$ is a strictly monotone closed symplectic manifold
and $N\geq 2$ or $\dim\H_*(M)>n+1$. Then the collection of strongly
non-degenerate Hamiltonians with infinitely many geometrically
distinct simple periodic orbits is of second Baire category in the
space of all Hamiltonians with respect to the $C^\infty$-topology.
\end{Corollary}

\begin{proof} First, note that we can assume that $M$ is strictly
  positive monotone, for otherwise the Conley conjecture holds,
  \cite{GG:nm}. Then the argument relies on the fact that the
  resonance relations, \eqref{eq:res1}, can be broken by a
  $C^\infty$-small perturbation.  To be more precise, consider the set
  $\CH$ of strongly non-degenerate Hamiltonians $H$ such that for each
  $k$ all $k$-periodic orbits of $H$ have distinct augmented
  actions. This set is of second Baire category in the space of all
  Hamiltonians with respect to the $C^\infty$-topology; see
  \cite{GG:generic}. Furthermore, $\CH$ is closed under iterations,
  i.e., $\H^{\nat k}\in \CH$ whenever $H\in\CH$. Therefore, every
  $H\in \CH$ has infinitely many geometrically distinct simple
  periodic orbits. Indeed, if some $H\in\CH$ had only finitely many
  simple periodic orbits, a suitable iteration $H^{\nat k}\in\CH$
  would be perfect, which is impossible by Corollary \ref{cor:res1}.
\end{proof}

\begin{Remark}
  We are not aware of any example of a closed monotone symplectic
  manifold $M^{2n}$ not meeting the requirements of Corollaries
  \ref{cor:res1} and \ref{cor:gen}. Such a manifold would have $N=1$
  and $\H_*(M)=\H_*(\CP^n)$.
\end{Remark}

Corollary \ref{cor:gen} considerably broadens the class of symplectic
manifolds for which $C^\infty$-generic existence of infinitely many
periodic orbits is known. Further, it is clear that the corollary can
be extended to some rational symplectic manifolds meeting similar
requirements. However, the conjecture that this is true for all closed
symplectic manifolds still remains completely open.

With more information about the structure of the quantum product one
is sometimes able to find several simple periodic orbits with
augmented action equal to $\hc_\infty$.

\begin{Corollary}[Resonance Relations, II]
\label{cor:res2}
Let $H$ be a perfect Hamiltonian on a strictly positive monotone
symplectic manifold $M^{2n}$. Assume that
\begin{equation}
\label{eq:product}
w_0*w_1*\cdots *w_{\ell} =q^\nu w_0 \text{ in } \HQ_*(M), 
\end{equation}
where $\nu>0$, for some classes $w_0\in \H_*(M)$, and
$w_1,\ldots,w_{\ell}\in \H_{*<2n}(M)$, and
$$
2n(\ell -1) -|w_1|-\cdots -|w_{\ell-1}| < 2N.
$$
Assume furthermore that $\nu=1$ or $H$ is non-degenerate. Then there
exist $\ell$ distinct one-periodic orbits $x_0,\ldots,x_{\ell-1}$ such
that
$$
\tA_{H}(x_i)=\hc_\infty.
$$
\end{Corollary}

We emphasize that the main point of this result is not the existence
of $\ell$ distinct periodic orbits, which is well known, but the fact
that all these orbits have augmented action equal to $\hc_\infty$.

The corollary is a consequence of \cite[Thm.\ 1.1]{CGG} and its proof
and of Theorem \ref{thm:aa-c_infty}. Among the manifolds the corollary
applies to (with $\F=\Q$) are, e.g., complex Grassmannians and their
monotone products. Moreover, the conditions of the theorem are
satisfied for $M\times V$, where $V$ is symplectically aspherical,
once they are satisfied for $M$. (We refer the reader to \cite{CGG}
for a more detailed discussion.)

For $M=\CP^n$, Corollary \ref{cor:res2} takes the following
particularly simple form, refining \cite[Thm.\ 1.12]{GG:gaps}.

\begin{Corollary}
\label{cor:CPn}
Let $H$ be a perfect Hamiltonian on $\CP^n$. Then $H$ has $n+1$
distinct one-periodic orbits with augmented action equal to
$\hc_\infty$.
\end{Corollary}

\end{document}